\documentclass[12pt]{amsart}
\usepackage{amssymb}
\usepackage{amsmath}
\usepackage[dvips]{graphicx}
\usepackage{hyperref}
\usepackage[dvipsnames]{xcolor}
\usepackage{todonotes}
\usepackage{mathtools}
\usepackage{enumitem}
\usepackage{fullpage}
\usepackage{tikz}
\usetikzlibrary{arrows}
\usetikzlibrary{decorations.markings}
\usetikzlibrary{tikzmark,decorations.pathreplacing}
\usetikzlibrary{shapes.geometric}
\numberwithin{equation}{section}
\usepackage{float}
\usepackage[capitalize,nameinlink,noabbrev,nosort]{cleveref}

\newcommand{\ds}{\displaystyle}
\DeclareMathOperator{\Sp}{Sp}

\DeclareMathOperator{\GL}{GL}
\DeclareMathOperator{\SL}{SL}
\DeclareMathOperator{\OO}{SO}
\DeclareMathOperator{\oo}{so}
\renewcommand{\sp}{\mathrm{sp}}
\renewcommand{\OE}{\mathrm{O}}
\renewcommand{\oe}{\mathrm{o}^{\text{even}}}
\newcommand{\usp}{\mathfrak{sp}}
\newcommand{\uoo}{\mathfrak{so}}
\newcommand{\uoe}{\mathfrak{o}}
\newcommand{\uX}{X_{\infty}}

\renewcommand{\P}{\mathcal{P}}

\newcommand{\Ct}{\mathcal{C}_t}
\DeclareMathOperator{\sgn}{sgn}
\newcommand{\floor}[1]{\left\lfloor #1 \right\rfloor}

\newcommand{\core}[2]{\mathrm{core}_{#2}{(#1)}}
\newcommand{\quo}[2]{\mathrm{quo}_{#2}{(#1)}}
\newcommand{\x}{\bar{x}}
\newcommand{\X}{\overline{X}}
\newcommand{\tX}{\widetilde{X}}

\DeclareMathOperator{\rk}{rk}
\DeclareMathOperator{\rev}{rev}

\DeclareMathOperator{\hs}{hs}
\DeclareMathOperator{\rs}{rs}
\DeclareMathOperator{\gl}{gl}

\newtheorem{thm}{Theorem}[section]
\newtheorem{prop}[thm]{Proposition}
\newtheorem{cor}[thm]{Corollary}
\newtheorem{lem}[thm]{Lemma}

\newtheorem{rem}[thm]{Remark}

\newtheorem{eg}[thm]{Example}

\crefname{thm}{Theorem}{Theorems}

\title
{Further results for classical and universal characters twisted by roots of unity}

\author{Arvind Ayyer}
\address{Arvind Ayyer, Department of Mathematics, 
Indian Institute of Science, Bangalore  560012, India.}
\email{arvind@iisc.ac.in}

\author{Nishu Kumari}
\address{Nishu Kumari, Department of Mathematics, 
Indian Institute of Science, Bangalore  560012, India.}
\email{nishukumari@alum.iisc.ac.in}

\date{\today}

\dedicatory{To Vyjyanthi Chari on the occasion of her 65th birthday}

\begin{document}

\begin{abstract}
We revisit factorizations of classical characters under various specializations, some old and some new. We first show that all characters of classical families of groups twisted by odd powers of an even primitive root of unity factorize into products of characters of smaller groups. Motivated by conjectures of {Prasad and Wagh} (Manuscr. Math. {2022}), we then observe that certain specializations of Schur polynomials factor into products of two characters of other groups.
We next show, via a detour through hook Schur polynomials, that certain Schur polynomials indexed by staircase shapes factorize into linear pieces.
Lastly, we consider classical and universal characters specialized at roots of unity. One of our results, in parallel with Schur polynomials, is that universal characters take values only in $\{0, \pm 1, \pm 2\}$ at roots of unity.
\end{abstract}

\subjclass[2010]{05A15, 05E05, 05E10, 20G05, 20G20}
\keywords{classical characters, universal characters, twisting by roots of unity, factorizations, hook Schur polynomials.}

\maketitle

\section{Introduction}
The irreducible characters of the classical families of groups possess many remarkable properties. In this work, we focus on the factorization properties of these characters under different specializations building on our previous work~\cite{ayyer-kumari-2022,kumari-2024}. This subject has a venerable history and we refer to Albion~\cite{albion-2025} for a detailed account.

Littlewood and Richardson were certainly one of the first who realized that the characters of the general linear group $\GL(n)$, also known as \emph{Schur polynomials} and denoted $s_\lambda(x_1, \dots, x_n)$, specialize to $-1, 0$ {or} $1$ when the variables are taken to be roots of unity~\cite[Theorem IX]{littlewood-richardson-1934}; see \cref{thm:roots of unity}. In the same work, they also generalized this to a version with many variables~\cite[Theorem VI]{littlewood-richardson-1934}; see \cref{thm:schur-fac}. 
This was later published in Littlewood's book~\cite[Chapter 7]{littlewood-1950}.
This latter result has been independently rediscovered by Prasad~\cite{prasad-2016}, and another proof has recently been given by Karmakar~\cite{karmakar-2023}. Kumar~\cite{kumar-2023} has generalized these results to flagged skew Schur polynomials.

In an almost unnoticed work, Lecouvey~\cite{lecouvey-2009b} generalized these classical character factorizations, which were found independently by the authors~\cite{ayyer-kumari-2022}. Kumari has generalized these factorizations with more specializations~\cite{kumari-2024}. 
Lecouvey~\cite{lecouvey-2009a}, in another little-known work, also generalized these results to universal characters of other families of Weyl groups of types B, C and D, and these were again found independently by Albion~\cite{albion-2023}. 
Very recently, Karmakar~\cite{karmakar-2024} has considered the classical characters evaluated at elements of order 2, generalizing a special case of these factorization results, and
{Nadimpalli--Pattanayak--Prasad~\cite{nadimpalli-et-al-2025} have given representation-theoretic proofs of some of these results.}

In another direction, specializations of Schur polynomials when one half of the variables are specialized to reciprocals of the other half were considered by Okada~\cite{okada-1998} and by Ciucu--Krattenthaler~\cite{CiuKra09} for rectangular partitions, and by Behrend--Fischer--Konvalinka~\cite{BehFisKon17} and Ayyer--Behrend--Fischer~\cite{AyyBehFis16} for double staircase partitions. These were generalized to self-dual shapes by Ayyer--Behrend~\cite{ayyer-behrend-2019} and skew shapes by Ayyer--Fischer~\cite{ayyer-fischer-2020}.

In this work, we extend several of these factorization results. 
{Our motivation is to find new factorizations and document some of those
not yet written down in the literature.}
We first consider the factorizations of classical characters twisted by odd powers of an even primitive root of unity in \cref{sec:fact}.
{That this special case would have nice factorizations was not clear beforehand, and this may be a case that may be amenable to representation-theoretic tools.}
In \cref{sec:rels}, we prove {unexpected} relations among characters of different groups when specialized to roots of unity. 
{Although most of the factorizations here are not new, what is new is that they are related to each other in some very simple way.}
In \cref{sec:staircase}, we begin by discussing when a universal character of classical type is independent of variables which are twisted by roots of unity. 
{Here, we note a phenomenon we have not seen before, namely that certain symmetric polynomials become independent of some of the variables.}
As a consequence, we provide necessary and sufficient conditions for a Schur polynomial in two finite sets of variables to be equal to the corresponding hook Schur polynomial. 
This will also prove factorization results for Schur polynomials indexed by staircase shapes. 
Finally, in \cref{sec:univ roots of unity} and \cref{sec:class roots of unity}, we will give 
explicit formulas for the universal and classical characters evaluated at roots of unity respectively.
{While it is true that these latter results are special cases of earlier older results, it is helpful to write them down separately. 
For instance, it is not apriori clear that $\uoe_\lambda$ evaluated at roots of unity is either $0$ or $\pm 1$ (\cref{thm:uoe roots of unity}) .}

We begin by reviewing the background material in \cref{sec:back}, which includes the relevant basics of partitions, formulas for classical and universal characters, and factorizations thereof.

\section{Background}
\label{sec:back}

\subsection{Partitions, cores and quotients}
We quickly recall the basic terminology of partitions as well as cores and quotients. 
For more details, see our earlier work~\cite{ayyer-kumari-2022}.

An \emph{(integer) partition} $\lambda = (\lambda_1, \dots, \lambda_n)$ is a weakly decreasing sequence of nonnegative integers.
The elements of the partition are called \emph{parts}. 
Let $\P$ be the set of all partitions.
The \emph{size} $|\lambda|$ of a partition $\lambda$ is the sum of its parts, and the \emph{length} $\ell(\lambda)$ of $\lambda$ is the number of nonzero parts. 
Let $\P_n$ be the set of all partitions with length at most $n$. {For $\lambda \in \P_n$ and $a \geq 0$, we denote the partition $(a + \lambda_1,\dots, a + \lambda_n)$ by $a + \lambda$. }

The \emph{conjugate} of a partition $\lambda$, denoted $\lambda'$, is the partition such that $\lambda'_i$ is the number of parts of $\lambda$ that are at least as large as $i$. For example, $(4, 2, 2, 1, 1)' = (5, 3, 1, 1)$. The Young diagram of $\lambda$ is a left-justified array of boxes where the $i$'th row contains $\lambda_i$ cells. Then the Young diagram of $\lambda'$ is obtained by transposing the rows and columns of that of $\lambda$. A partition $\lambda$ is said to be \emph{self-conjugate} if $\lambda = \lambda'$. {For partitions $\lambda, \mu \in \P$, we write
$\mu \subset \lambda$ if the Young diagram of $\mu$ is contained in that of $\lambda$.}

It is well-known that irreducible (polynomial) representations of $\GL(n)$ are indexed by partitions in $\P_n$.
Suppose $\lambda \in \P_n$, with $\lambda_1 = c$ indexes an irreducible representation of $\GL(n)$. Then the \emph{dual representation} is indexed by the partition $\hat{\lambda} =  (c - \lambda_n, \dots, c - \lambda_1)$. In other words, $\hat{\lambda}$ is obtained by considering the complement of the Young diagram of $\lambda$ inside the $c \times n$ rectangle, and reading it in reverse order. 
For example, if $\lambda = (5, 2, 1, 1) \in \P_4$, then $c = 5$ and so $\hat{\lambda} = (4, 4, 3, 0)$. If $\lambda$ is such that $\hat{\lambda} = \lambda$, then $\lambda$ is said to index a \emph{self-dual representation}. Notice that duality is tied to {the rank} of the underlying group. For example, $(4, 2)$ is self-dual for $\GL(3)$, but not for $\GL(4)$.

The \emph{(Frobenius) rank} of a partition $\lambda$, denoted $\rk(\lambda)$, is the largest integer $k$ such that $\lambda_k \geq k$.
A partition is said to be \emph{strict} if all its {nonzero} parts are distinct. 
Any partition can be expressed in \emph{Frobenius coordinates}, written $\lambda = (\alpha \mid \beta)$, where $\alpha$ and $\beta$ are strict partitions given by $\alpha_i = \lambda_i - i$ and $\beta_i = \lambda'_i - i$ for $i \in \{1,2,\dots,\rk(\lambda)\}$.
A partition $\lambda$ is said to be \emph{symplectic} (resp. \emph{orthogonal}) if $\lambda=(\alpha \mid \alpha+1)$ (resp. $\lambda=(\alpha+1 \mid \alpha)$) in Frobenius coordinates, for some strict partition $\alpha$. Notice that $\lambda$ is self-conjugate if and only if $\lambda = (\alpha \mid \alpha)$.

For a partition $\lambda$ and an integer $m$ such that $\lambda \in \P_m$, 
define the {\em {beta set} of $\lambda$} by $\beta(\lambda,m) = (\beta_1(\lambda,m),\dots,\beta_m(\lambda,m))$ where $\beta_i(\lambda,m) = \lambda_i+m-i$. 
{Let $t > 2$ be a fixed positive integer throughout the article.}
{For $0 \leq i \leq t-1$, suppose 
$n_{i}(\lambda,m)$ is the number of parts of $\beta(\lambda, m)$ congruent to $i \pmod{t}$}.
We will write $\beta(\lambda)$ and $n_{i}(\lambda)$ whenever $m$ is clear from the context. 

For a partition $\lambda \in \P_{tn}$, 
let $\sigma_{\lambda} \in S_{tn}$ be the permutation that rearranges the parts of {$\beta(\lambda) = \beta(\lambda, tn)$} such that
\begin{equation}
\label{sigma-perm}
 \beta_{\sigma_{\lambda}(j)}(\lambda) \equiv a \pmod t, \quad  \sum_{i=0}^{a-1} n_{i}(\lambda)+1 \leq j \leq \sum_{i=0}^{a} n_{i}(\lambda),   
\end{equation}
arranged in decreasing order for each $a \in \{0,1,\dots,t-1\}$. For example, consider $\lambda=(5,2,2,1,1) \in \P_6$ with $t=3$ and $n=2$. Then $\beta(\lambda)=(10,6,5,3,2,0)$ and therefore, $n_0(\lambda) = 3, n_1(\lambda) = 1, n_2(\lambda) = 2$, and 
$\sigma_{\lambda}=246135$ in one-line notation.

Macdonald defines the $t$-core and
$t$-quotient of $\lambda$ using its beta set.  
Suppose $\beta_j^{(i)}(\lambda)$, $1 \leq j \leq n_{i}(\lambda)$ {are} the $n_{i}(\lambda)$ parts of $\beta(\lambda)$ congruent to $i \pmod t$ in decreasing order.
Arrange the $tn$ numbers $tj+i$, where $0 \leq j \leq n_{i}(\lambda)-1$ and $0 \leq i \leq t-1$ in descending order, say $\tilde{\beta}_1>\dots>\tilde{\beta}_{tn}$. 
Then the \emph{$t$-core of} $\lambda$ has parts $\tilde{\beta}_i-tn+i$, $1 \leq i \leq tn$, and is denoted by $\core{\lambda}t$.   
Partitions $\lambda$ that satisfy $\core{\lambda}t = \lambda$ are called \emph{$t$-cores}. 
Let $\Ct$ denote the set of all $t$-cores.

The \emph{$t$-quotient} of $\lambda$ is a certain $t$-tuple of partitions denoted $\quo{\lambda}t = (\lambda^{(0)}, \dots, \lambda^{(t-1)})$, where the $(i+1)$'th element 
is obtained as follows. Write $\beta_j^{(i)}(\lambda)$ in the form $t\tilde{\beta}_j^{(i)}+i$, for $1 \leq j \leq n_{i}(\lambda)$, so that $\tilde{\beta}_1^{(i)}>\dots>\tilde{\beta}_{n_{i}(\lambda)}^{(i)}\geq 0$ and 
define the $i$'th element of the $j$'th tuple of the $t$-quotient as
$\lambda_j^{(i)}=\tilde{\beta}_j^{(i)}-n_{i}(\lambda)+j$. 
Note that the effect of changing $m\geq \ell(\lambda)$ to evaluate $\beta(\lambda)$ is to permute the $\lambda^{(j)}$ cyclically, so that $\quo{\lambda}t$ should perhaps be thought of as a `necklace' of partitions. 
An important fact is that the map from $\P \to \Ct \times \P^{\times t}$ given by
\[
\lambda \mapsto (\core{\lambda}t; \lambda^{(0)}, \dots, \lambda^{(t-1)})
\]
is a bijection. This result is sometimes called the \emph{partition division theorem}~\cite[Theorem 11.22]{loehr-2011},
and the bijection is called the \emph{Littlewood bijection}.
One can also obtain the $t$-core and $t$-quotient directly from its Young diagram; see~\cite{loehr-2011}.

For a partition $\lambda$, let $\rev(\lambda)$ denote the tuple formed by reversing the parts of $\lambda$ and $-\lambda$ denote the tuple whose entries are negatives of those of $\lambda$. Further, for $a \in \mathbb{N}$, let  $a + \lambda$ denote the partition formed by adding $a$ to every part of $\lambda$. Given partitions $\lambda = (\lambda_1, \dots, \lambda_m)$ and $\mu = (\mu_1, \dots, \mu_n)$ such that $\lambda_m \geq \mu_1$, we denote by $(\lambda, \mu)$ the partition obtained by concatenating the parts of $\lambda$ and $\mu$ in that order. Building on this notation, we will often need, for $\lambda, \mu \in \P$ the partition $\mu_1 + (\lambda, 0, \dots, 0, - \rev(\mu))$, where the number of $0$'s is arbitrary.
We will use the notation $(\lambda, -\mu)_n$ to denote such a partition, where the total number of $0$'s is $n - \ell(\lambda) - \ell(\mu)$. 
Note that $n$ must be at least as large as $\ell(\lambda) + \ell(\mu)$.
For example 
\[
((4, 1), -(2, 1, 1))_7 = 2 + (4, 1, 0, 0, -1, -1, -2) = (6, 3, 2, 2, 1, 1, 0).
\]
Notice that the last element of $(\lambda, -\mu)_n$ is always $0$.
We will also use the notation $(\pm \lambda)_n = (\lambda, -\lambda)_n$ often, which is self-dual for all $n \geq 2 \ell(\lambda)$.

\subsection{Classical characters}
The classical {families} of simple Lie groups are classified into the general linear groups $\GL(n)$, the orthogonal groups $\OO(n)$ and the symplectic groups $\Sp(2n)$. The general linear groups are said to belong to type A, the orthogonal groups of odd size, to type B, the symplectic groups, to type C, and the orthogonal groups of even size, to type D. In each of these cases, the irreducible characters are symmetric polynomials or Laurent polynomials in certain variables, and are indexed by partitions.

To define the characters, we need some notation. Fix a positive integer $n$ and
let $X=(x_1,x_2,\dots,x_n)$ be a fixed $n$-tuple of commuting indeterminates throughout. Set $\x = 1/x$ for every commuting indeterminate $x$ and $\X = (\x_1, \dots, \x_n)$. 
Similarly, define $X^m = (x_1^m, \dots, x_n^m)$ for any $m \in \mathbb{Z}$.

We recall that the \emph{complete (homogeneous) symmetric polynomial} indexed by a nonnegative integer $m$ is given by
\begin{equation}
\label{def-hm}
h_m(X) = 
\sum_{1 \leq i_1 \leq i_2 \leq \cdots \leq i_m \leq n} x_{i_1} x_{i_2} \dots x_{i_m},
\end{equation}
where we set $h_0(X) = 1$ and $h_m(X) = 0$ if $m < 0$. 
It is easy to see that $h_m(-X)=(-1)^m h_m(X)$. 
It is a standard fact~\cite[Equation (2.5)]{macdonald-2015} that the ordinary generating function for the complete symmetric polynomials is
\begin{equation}
\label{genfn-h}
\sum_{m \geq 0} h_m(X) q^m= \ds \prod_{i=1}^n  \frac{1}{(1-x_i q)}. 
\end{equation}

\begin{prop} 
\label{prop:compid}
We have, for $m \in \mathbb{N}$,
\begin{equation}
    h_m(X,-1)+h_{m-1}(X,-1)=h_m(X).
\end{equation}    
\end{prop}

\begin{proof}
The {ordinary} generating function of the left hand side is
\begin{align*}
\sum_{m \geq 0} h_m(X,-1) q^m + \sum_{m \geq 0} h_{m-1}(X,-1) q^m= &
\ds \prod_{i=1}^n  \frac{1}{(1+q)(1-x_i q)}
+ \ds \prod_{i=1}^n  \frac{q}{(1+q)(1-x_i q)},
\end{align*}
which simplifies to that of the right hand side using \eqref{genfn-h}.
\end{proof}

For a partition $\lambda=(\lambda_1,\dots,\lambda_n) \in \P_n$,
the \emph{Schur polynomial} or \emph{$\GL(n)$ character} is given by
\begin{equation}
    \label{def:SS-Jcob}
    s_{\lambda}(X) \coloneqq \det\left(h_{\lambda_i-i+j}(X)\right)_{1 \leq i,j \leq n}.
\end{equation}
This is known as the \emph{Jacobi-Trudi identity}, and is 
one of many formulas for the Schur polynomial. See~\cite{stanley-ec2} for other formulas.{
For partitions $\lambda, \mu$ such that $\mu \subset \lambda$, 
the \emph{skew Schur polynomial} is given by
\begin{equation}
    \label{def:SSkew-Jcob}
    s_{\lambda/\mu}(X) \coloneqq \det\left(h_{\lambda_i-\mu_j-i+j}(X)\right)_{1 \leq i,j \leq n}.
\end{equation}
{Note that $s_{\lambda/\mu}$ is zero if $\mu \nsubseteq \lambda$.}
Littlewood's expansion~\cite[Equation (2.9)]{albion-2023} says that
\begin{equation}
\label{littexp}
s_{(\lambda,-\mu)_n}(X)= \sum_{\nu \in \P} (-1)^{|\nu|} s_{\lambda/\nu}(X) s_{\mu/\nu'}(X).
\end{equation}
}

The remaining classical characters are Laurent polynomials in $x_1, \dots, x_n$ that are symmetric both under the exchange of $x_i$ and $x_j$ for $1 \leq i < j \leq n$ as well as $x_i$ and $1/x_i$ for any $i \in \{1,2,\dots,n\}$.
The \emph{odd orthogonal (type B) character} of the group $\OO(2n+1)$ at the representation indexed
by $\lambda=(\lambda_1,\dots,\lambda_n)$ is given by
\begin{equation}
 \label{def:odd-jcob}
  \oo_{\lambda}(X)  = \det \left( h_{\lambda_i-i+j}(X,\X,1)-h_{\lambda_i-i-j}(X,\X,1) \right)_{1 \leq i,j \leq n}.   
\end{equation}
The \emph{symplectic (type C) character} of the group $\Sp(2n)$ at the representation indexed
by $\lambda=(\lambda_1,\dots,\lambda_n)$ is given by
\begin{equation}
\label{def:sp-jcob}
    \sp_{\lambda}(X)
 = \frac{1}{2} \det \left(
     h_{\lambda_i-i+j}(X,\X)+ h_{\lambda_i-i-j+2}(X,\X)  
 \right)_{1 \leq i \leq n}.
\end{equation}
Lastly, the \emph{even orthogonal (type D) character} of the group $\OE(2n)$ at the representation indexed
by $\lambda=(\lambda_1,\dots,\lambda_n)$ is given by
\begin{equation}
\label{def:even-Jcob}
  \oe_{\lambda}(X)=\det \left( h_{\lambda_i-i+j}(X,\X)-h_{\lambda_i-i-j}(X,\X) \right)_{1 \leq i,j \leq n}.  
\end{equation}

\subsection{Universal characters}

The Schur polynomial is itself a universal character because the $\GL(n)$ character $s_\lambda(X)$, when $\ell(\lambda) \leq n-1$ and $x_n$ is set to $0$, becomes the $\GL({n-1})$ character indexed by the same partition $\lambda$. This is not true for the other characters because they are Laurent polynomials in the $x_i$'s and so we cannot set $x_n$ to $0$ there.

Koike and Terada~\cite{koike-terada-1990} defined \emph{universal characters} {(see below)} to address this issue for classical groups of other types. These characters are symmetric polynomials (not Laurent {polynomials}) which reduce correctly.

Suppose $\uX=(x_1,x_2,\dots)$. 
Koike and Terada~\cite{koike-terada-1990} defined the universal character for $\OE({n})$ as
\begin{equation}
\label{def-univo}
\uoe_\lambda(\uX) 
=\det \left(
     h_{\lambda_i-i+j}(\uX)- h_{\lambda_i-i-j}(\uX)  
 \right)_{{1 \leq i,j \leq n}},
\end{equation}
and the universal character for $\Sp({2n})$,
\begin{equation}
\label{def-univsp}
\usp_\lambda(\uX) 
=\frac{1}{2} \det \left(
     h_{\lambda_i-i+j}(\uX)+ h_{\lambda_i-i-j+2}(\uX)  
 \right)_{{1 \leq i,j \leq n}}.
\end{equation}
They then showed that if $\lambda \in \P_n$,
\begin{equation}
\label{uni-class}
\begin{split}
\oe_\lambda(X) = & 
 \uoe_\lambda \left(X, \X, 0, 0, \dots \right), \\
\oo_\lambda(X) = & 
 \uoe_\lambda \left(X, \X, 1, 0, 0, \dots \right), \\
\sp_\lambda(X) = & 
 \usp_\lambda \left(X, \X, 0, 0, \dots \right).
\end{split}
\end{equation}

Koike~\cite[Definition 6.4]{koike-1997} later defined the universal character for the group $\OO({2n+1})$. If {$\lambda \in \P_n
$}, he set
\begin{equation}
\label{def-univso}
\uoo_\lambda(\uX) 
=\det \left(
     h_{\lambda_i-i+j}(\uX)+ h_{\lambda_i-i-j+1}(\uX)  
 \right)_{{1 \leq i,j \leq n}},
\end{equation}
and showed that $\oo_{\lambda}(X)=\uoo_\lambda(X, \X, 0 , 0 ,\dots)$.
Albion~\cite[Section 2.3]{albion-2023} {defined} a variant of the above as
\begin{equation}
\label{soneg-1}
\uoo_\lambda^-(\uX) 
=\det \left(
     h_{\lambda_i-i+j}(\uX)- h_{\lambda_i-i-j+1}(\uX)  
 \right)_{{1 \leq i,j \leq n}},
\end{equation}
and {proved} that
\begin{equation}
\label{soneg}
{\uoo_\lambda^-(\uX)=(-1)^{|\lambda|} \uoo_\lambda(-\uX).}
\end{equation}
{
For partitions $\lambda$ and $\mu$, Koike~\cite{koike1989decomposition} defined a universal character  
of $\GL(n)$ by
\begin{equation}
\label{defrs}
\rs_{\lambda,\mu}(\uX)=\sum_{\nu \in \P} (-1)^{|\nu|}s_{\lambda/\nu}(\uX) s_{\mu/\nu'}(\uX).
\end{equation}
}
If $\ell(\lambda)+\ell(\mu) \leq n$, then by \eqref{littexp}, it is immediate that
\begin{equation}
\label{rs-s}
\rs_{\lambda,\mu}(\uX)=s_{(\lambda,-\mu)_n}(\uX).
\end{equation}
In what follows, it should be understood that a universal character with finitely many variables means that all the other variables are set to $0$.

\subsection{Useful factorizations}

Here we write down all the factorization results we use later in the article. We first begin with those for $\GL(n)$ characters, i.e. Schur polynomials.
Recall that $t \geq 2$ is a positive integer and $\omega$ is a primitive $t$'th root of unity. 
The evaluation of the Schur polynomials at the roots of unity is astonishingly simple.

\begin{thm}[{\cite[Theorem IX]{littlewood-richardson-1934}, \cite[Chapter I.3, Example 17(a)]{macdonald-2015}}]
\label{thm:roots of unity}
Let $\lambda \in \P_t$. Then
\[
s_\lambda(1, \omega, \omega^2, \dots, \omega^{t-1}) = \begin{cases}
(-1)^{\frac{t(t-1)}{2}} \sgn(\sigma_{\lambda}) & \core{\lambda}t = \emptyset, \\
0 & \text{otherwise}.
\end{cases}
\]
\end{thm}

We will generalize \cref{thm:roots of unity} to other universal characters in \cref{sec:univ roots of unity} and to classical characters in \cref{sec:class roots of unity}.
The generalization of this result by adding extra variables is classical. A proof in modern language is given in~\cite{ayyer-kumari-2022}.

\begin{thm}[{\cite[Theorem VI]{littlewood-richardson-1934}, \cite[Theorem 2]{prasad-2016}}]
\label{thm:schur-fac}
Let $\lambda \in \P_{tn}$ be a partition  
indexing an irreducible representation of $\GL({tn})$ and $\quo \lambda t  = (\lambda^{(0)},\dots,\lambda^{(t-1)})$.
Then the $\GL({tn})$-character $s_{\lambda}(X,\omega X, \dots ,\omega^{t-1}X)$ is nonzero if and only if $\core \lambda t$ is {empty}. In that case, it is given by 
\begin{equation}
  \label{th:1.2} s_{\lambda}(X,\omega X, \dots ,\omega^{t-1}X) \\ =  (-1)^{\frac{t(t-1)}{2}\frac{n(n+1)}{2}}\sgn(\sigma_{\lambda})\prod_{i=0}^{t-1} s_{\lambda^{(i)}}(X^t).
  \end{equation}
\end{thm}

A further generalization of \cref{thm:schur-fac} is as follows.
Let $E = (e_1, \dots, e_m)$ be such that $t-1 \geq e_1 > \cdots > e_m \geq 0$. Define $e_{m+1}, \dots, e_t$ to be the elements of $\{0, \ldots, t-1\} \setminus \{e_1, \dots, e_m\}$ in increasing order.
For a partition $\lambda \in \P_{tn+m}$,  
let $\sigma_{\lambda}^{E}$ be the permutation in $S_{tn+m}$ such that it rearranges parts of $\beta(\lambda)$ in the following way: 
\begin{equation}
\label{sigma-perm-m}
 \beta_{\sigma^E_{\lambda}(j)}(\lambda) \equiv e_q \pmod t, \quad  \text{ for } \quad \sum_{i=1}^{{q-1}} n_{e_i}(\lambda)+1 \leq j \leq \sum_{i=1}^{q} n_{e_i}(\lambda),   
\end{equation}
arranged in decreasing order for each $q \in \{1,\dots,t\}$. For simplicity, we write $\sigma_{\lambda}^{\emptyset}$ as $\sigma_{\lambda}$.

\begin{thm}[{\cite[Theorem 2.2]{kumari-2024}}]
\label{thm:schur-k}
Fix $0 \leq m \leq t-1$. Let $\lambda$ be a partition in $\P_{tn+m}$.
Then the $\GL({tn+m})$-character $s_{\lambda}(X, \omega X,\dots,\omega^{t-1} X, y,\omega y, \dots, \omega^{m-1}y )$ is nonzero if and only if $\core \lambda t= \nu \coloneqq (\nu_1,\dots,\nu_m)$ with $\nu_1\leq t-m$. In that case,
    \begin{equation}
    \label{schur-k}
    \begin{split}
   s_{\lambda}(X, \omega X,\dots,\omega^{t-1} X,  y,\omega y, \dots,  \omega^{m-1}y) 
       =&  \sgn(\sigma_{\lambda}^{\beta(\nu)})
       \sgn(\sigma_{\emptyset}^{\beta(\nu)}) 
        s_{\nu}(1,\omega, \dots,\omega^{m-1})\\
    & \times   
       \prod_{i=1}^{m} s_{\lambda^{(\beta_i(\nu))}} (X^t,y^t) \prod_{\substack{j=0\\ j \neq \beta_i(\nu),1 \leq i \leq m}}^{t-1}  s_{\lambda^{(j)}} (X^t),    
        \end{split}
    \end{equation}
where we recall that the $t$-quotient of $\lambda$ is $(\lambda^{(0)}, \dots, \lambda^{(t-1)})$.
\end{thm}

We now move on to factorization of universal characters
of other types.

\begin{thm}[{\cite[Theorem 3.3]{albion-2023}}]
\label{thm:sympfact} 
Let $\lambda$ be a partition in $\P_{tn}$.
Then the universal character ${\usp}_{\lambda}(\uX,\omega \uX, \allowbreak \dots ,\omega^{t-1}\uX)$ is nonzero if and only if $\core \lambda t$ is symplectic. In that case, if $\core \lambda t$ has rank $r$, then 
\begin{equation}
    \begin{split}
      \label{symp-fact-2}
{\usp}_{\lambda}(\uX,\omega \uX, \dots ,\omega^{t-1}\uX) =  (-1)^{\epsilon} \sgn (\sigma_{\lambda})\, {\usp}_{\lambda^{(t-1)}}(\uX^t) 
 & \prod_{i=0}^{\floor{{(t-3)}/{2}}} 
 \rs_{\lambda^{(i)},\lambda^{(t-2-i)}}(\uX^t) \\ & \times 
\begin{cases}
{\uoo}_{\lambda^{({(t-2)}/{2})}}(\uX^t) 
& t \text{ even} ,\\
1 & t \text{ odd},
\end{cases}  
    \end{split}
\end{equation}
where 
\[
\epsilon= \ds -\sum_{i=\floor{{t}/{2}}}^{t-2} 
\binom{n_{i}(\lambda)+1}2 + \begin{cases}
\frac{n(n+1)}{2}+nr & t \text{ even},
\\0  & t \text{ odd}.
\end{cases}
\]
\end{thm}

\begin{thm}[{\cite[Theorem 3.2]{albion-2023}}]
\label{thm:eorthfact}
Let $\lambda$ be a partition in $\P_{tn}$.
Then the universal character $\uoe_{\lambda}(\uX,\omega \uX, \dots ,\omega^{t-1}\uX)$ 
is  nonzero if and only if $\core \lambda t$ is orthogonal. In that case, if $\core \lambda t$ has rank $r$, then 
\begin{equation}
    \begin{split}
     \label{th:5.2}  
     \uoe_{\lambda}(\uX,\omega \uX, \dots ,\omega^{t-1}\uX)  =  (-1)^{\epsilon}
     \sgn
(\sigma_{\lambda}) \,
 & \uoe_{\lambda^{(0)}}(\uX^t)   
 \prod_{i=1}^{\floor{{(t-1)}/{2}}} 
 \rs_{\lambda^{(i)},\lambda^{(t-i)}}(\uX^t) \\&  \times 
\begin{cases}
\uoo^{-}_{\lambda^{(t/2)}}(\uX^t) & t \text{ even},\\
1 & t \text{ odd},
\end{cases}   
    \end{split}
\end{equation}
where 
\[
\epsilon=-\sum_{i=\floor{{(t+2)}/{2}}}^{t-1} 
\binom{n_{i}(\lambda)+1}{2} +r+ \begin{cases}
\frac{n(n+1)}{2}+nr & t \text{ even}, \\ 
0 & t \text{ odd}.
\end{cases}
\]

\end{thm}

\begin{thm}[{\cite[Theorem 2.13]{kumari-2024}}]
\label{thm:evenfac-1}
    Let $\lambda \in \P_{tn+1}$ 
    such that $\core \lambda t$ is symplectic. The even orthogonal character $\oe_{\lambda}(X,\, \omega X, \dots,\omega^{t-1}X,1)$ is given by 
    \begin{equation}
    \label{even--1}
    \begin{split}
   \oe_{\lambda}(X,\omega X, \dots ,\omega^{t-1}X,1) =
 (-1)^{\epsilon} \, \sgn(\sigma_{\lambda}) & \, \oe_{\lambda^{(0)}}(X^t,1)  
 \prod_{i=1}^{\floor{{(t-1)}/{2}}} 
 s_{(\lambda^{(i)}, -\lambda^{(t-i)})_{2n}}(X^t,\X^t) 
 \\
 \times & \begin{cases}
(-1)^{\sum_i \lambda_i^{(t/2)}} \oo_{\lambda^{(t/2)}}(-X^t)& t \text{ even,}\\
1 & t \text{ odd,}
\end{cases}        
\end{split}
    \end{equation}
where     
 \begin{equation*}
\epsilon =
\ds \sum_{i=\floor{{(t+2)}/{2}}}^{t-1}  \left( \binom{n_i(\lambda)}{2} +
(t-1) n (n_{i}(\lambda)-n) \right).
\end{equation*}
\end{thm}

\begin{thm}[{\cite[Theorem 3.4]{albion-2023}}]
    \label{thm:oorthfact}
Let $\lambda$ be a partition in $\P_{tn}$. 
Then the universal character $\uoo_\lambda(\uX,\omega \uX, \dots ,\omega^{t-1}\uX)$ is  nonzero if and only if $\core \lambda t$ is self-conjugate. In that case, if $\core \lambda t$ has rank $r$, then 
\begin{equation}
   \begin{split}
     \label{th:3.2} \uoo_\lambda(\uX,\omega \uX, \dots ,\omega^{t-1}\uX)  =  (-1)^{\epsilon}\sgn(\sigma_{\lambda}) 
\prod_{i=0}^{\floor{{(t-2)}/{2}}} 
&  \rs_{\lambda^{(i)},\lambda^{(t-1-i)}}(\uX^t) \\ & 
\times 
\begin{cases}
\uoo_{\lambda^{({(t-1)}/{2})}}(\uX^t) 
& t \text{ odd},\\
1 & t \text{ even},
\end{cases}   
    \end{split}
\end{equation}
where 
\[
\epsilon=- \ds \sum_{i=\floor{{(t+1)}/{2}}}^{t-1} \binom{n_{i}(\lambda)+1}{2} + \begin{cases}
nr & t \text{ odd},
\\ 0 & t \text{ even}.
\end{cases}
\]

\end{thm}

We note that specializing the variables to $(X,\X)$ in \cref{thm:sympfact}, \cref{thm:eorthfact} and \cref{thm:oorthfact} gives corresponding factorization results for classical characters~\cite{ayyer-kumari-2022}. 
The next two results are factorizations for characters of self-dual representations of $\GL(n)$ when half the variables are reciprocals of the other half.

\begin{thm}[{\cite[Equation (18)]{ayyer-behrend-2019}}]
\label{thm:facx}
    For any partition $\lambda=(\lambda_1,\dots,\lambda_n)$,
    \[
    s_{(\pm \lambda)_{2n}}(X,\X)=(-1)^{ 
    |\lambda|} 
    \oo_{\lambda}(X) \,
    \oo_{\lambda}(-X).
    \]
\end{thm}

The following result is a reformulation of~\cite[Equation (12)]{ayyer-behrend-2019} when $k=\lambda_1$ and $\lambda_n=0$. 

\begin{thm}
\label{thm:facx1}
For any partition $\lambda=(\lambda_1,\dots,\lambda_n)$,
\[
s_{(\pm \lambda)_{2n+1}}(X,\X,1)=\sp_{\lambda}(X) \, \oe_{\lambda}(X,1).
\]
\end{thm}

\section{Factorization of classical characters twisted by odd powers of an even primitive root of unity}
\label{sec:fact}

We generalize the results from our previous work~\cite{ayyer-kumari-2022} in this section.
Suppose $t$ is a multiple of $4$ and $\omega$ is a primitive $t'$th root of unity as usual. 
Define  
a tuple $X_{\omega}$ of length $tn/2 + t/4$ by 
\[
X_{\omega} = (\omega X,\omega^3 X,\dots,\omega^{{t}-1}X,\omega,\omega^{3},\dots,\omega^{t/2-1}){,}
\]
and let $\X_{\omega}=\overline{(X_{\omega})}$ so that
\[
\X_{\omega} = (\omega \X,\omega^3 \X,\dots,\omega^{{t}-1}\X,\omega^{t/2+1},\omega^{t/2+3},\dots,\omega^{t-1}).
\]
It then follows that
    \begin{align*}
 (X_{\omega},\X_{\omega}) 
= (Z,\omega^2 Z,\dots, \omega^{t-2} Z),
    \end{align*}
 where  $Z=(\omega X, \omega \X,\omega)$.
We will evaluate classical characters specialized at $X_\omega$.
For $n = 1, X=(x)$ and  $t=4$, the specialization considered here is $X_\omega = (\iota x, -\iota x, \iota)$, where $\iota^2 = -1$.
The motivation for considering this strange specialization comes from {a suggestion of D. Prasad, who was looking for natural generalizations of \cref{thm:schur-k}. }

\subsection{Schur factorization}

\begin{thm}
\label{thm:schur-k-s}
Let $\lambda \in \P_{{tn}/{2}+{t}/{4}}$ be a partition 
indexing an irreducible representation of $\GL({tn}/{2}+{t}/{4})$,
with $\core \lambda {{t}/{2}}= \nu \coloneqq (\nu_1, \nu_2, \dots)$ 
and $\quo \lambda t = (\lambda^{(0)},\dots,\lambda^{(t-1)})$. 
Then the $\GL({tn}/{2}+{t}/{4})$-character $s_{\lambda}(X_{\omega})$ is
nonzero if and only if $\ell(\nu) \leq t/4$ and $\nu_1\leq t/4$. 
In that case, it is given by
    \begin{equation}
    \label{schur-k-s}
       s_{\lambda}(X_{\omega})
   = \sgn(\sigma_{\lambda}^{\beta(\nu)})
       \sgn(\sigma_{\emptyset}^{\beta(\nu)}) \, s_{\nu}(1,\omega^{2},\dots,\omega^{t/2-2})
       \prod_{i=1}^{{t}/{4}} 
       s_{\lambda^{(\beta_i(\nu))}}
       (X^{{t}/{2}},1) 
       \prod_{\substack{j=0\\ j \not \in \beta(\nu)}}
       ^{(t-2)/{2}}  
       s_{\lambda^{(j)}} (X^{{t}/{2}}).
    \end{equation}
\end{thm}

\begin{proof}
This is a direct corollary of \cref{thm:schur-k}, where we use $t/2$, $\omega^2$ and $\omega X$ in place of $t$, $\omega$ and $X$ respectively, and substitute $m=t/4$ and $y=\omega$. Notice that in the non-zero case, one of the factors is 
the Schur polynomial 
\begin{equation}
\label{svan}
s_{\nu}(1,\omega^{2},\dots,\omega^{t/2-2}) = 
    \prod_{1 \leq i<j \leq \frac{t}{4}} \left(\frac{\omega^{2\nu_j+\frac{t}{2}-2j}
-\omega^{2\nu_i+\frac{t}{2}-2i}}
{\omega^{t/2-2j}
-\omega^{t/2-2i}}\right). 
\end{equation}
Therefore, \eqref{schur-k-s} is always non-zero, completing the proof. 
\end{proof}

\subsection{Symplectic character factorization}

\begin{lem}
\label{lem:relations}
We have $\uoo_\lambda(X,\X,-1) =\sp_{\lambda}(X)$ for $\lambda \in \P_{n}$, and 
${\usp}_{\lambda}(X,\X,-1) =\oo_{\lambda}(X,-1)$ for $\lambda \in \P_{n+1}$.
\end{lem}

\begin{proof} 
From \eqref{def-univso},
\[
\uoo_\lambda(X,\X,-1) 
=\det \left(
     h_{\lambda_i-i+j}(X,\X,-1)+ h_{\lambda_i-i-j+1}(X,\X,-1)  
 \right)_{1 \leq i,j \leq n}.
\]
We can now apply \cref{prop:compid} to simplify each entry.
We then apply elementary column operations $C_{j} \rightarrow C_j+C_{j-1}$ for $j=n,\dots,2$ in that order  
to obtain
\begin{align*}
\uoo_\lambda(X,\X,-1) 
&=\det \left(
\begin{array}{c|c}
     h_{\lambda_i-i+1}(X,\X) & \left(  h_{\lambda_i-i+j}(X,\X)+ h_{\lambda_i-i-j+2}(X,\X)  \right)_{2 \leq j \leq n}
\end{array}
 \right)_{1 \leq i \leq n}\\
 &=\sp_{\lambda}(X).
\end{align*}
This proves the first equality. For the second equality, consider ${\oo}_{\lambda}(X,-1) = \uoo_\lambda(X,\X,-1,\allowbreak -1)$ and repeat the above computation.
\end{proof}

We consider the symplectic character with the same specialization below. We first recall the following result. 

\begin{lem}[{\cite[Corollary 3.7]{ayyer-kumari-2022}}]
\label{len-res-smp}
    Let $\lambda$ be a partition of length 
at most $tn$. Then $\core \lambda {t}$ is a symplectic $t$-core if and only if
$n_{i}(\lambda)+n_{t-2-i}(\lambda)=2n$ for $0 \leq i \leq \floor{{(t-2)}/{2}}$ and $n_{t-1}(\lambda)=n$.
\end{lem}

\begin{thm} 
\label{thm:symp}
Let $\lambda$ be a partition in $\P_{{tn}/{2}+{t}/{4}}$.
The $\Sp(tn+{t}/{2})$-character $\sp_{\lambda}(X_{\omega})$ is 
nonzero if and only if $\core \lambda {t/2}$ is symplectic.
In that case, it is given by 
    \begin{equation}
    \label{sym-res}
   \begin{split}
  \sp_{\lambda}(X_{\omega})  =
  (-1)^{\epsilon} & \sgn({\sigma}_{\lambda}) \oo_{\lambda^{(t/2-1)}}(-X^{t/2},-1)  \\
  &\times 
  \sp_{{\lambda^{((t-4)/{4})}}}(-X^{t/2})
   \prod_{i=0}^{{(t-8)/{4}}} 
  s_{(\lambda^{(i)},-\lambda^{(t/2-2-i)})_{2n+1}}(X^{t/2},\X^{t/2}, 1),
  \end{split}
    \end{equation}
    where 
the $(t/2)$-quotient $(\lambda^{(0)},\dots,\lambda^{(t/2-1)})$ is evaluated by considering $\lambda \in \P_{tn+t/2}$ and 
\[
\epsilon= \ds -\sum_{i={{t}/{4}}}^{(t-4)/2} 
\binom{n_{i}(\lambda)+1}2 + 
{n+1}+r+\sum_{i=0}^{{(t-8)/{4}}} |(\lambda^{(i)},-\lambda^{(t/2-2-i)})_{2n+1}|.
\]
\end{thm}

{Notice that neither the odd orthogonal nor the symplectic factor in \eqref{sym-res} is a true character.}

\begin{proof}  
Since
$(X_{\omega},\X_{\omega}) = (Z,\omega^2 Z,\dots, \omega^{t-2} Z)$, we have from \eqref{uni-class},
    \[
    \sp_{\lambda}(X_{\omega})
  = {\usp}_{\lambda}(Z,\omega^2 Z,\dots, \omega^{t-2} Z).
    \]
If $\core \lambda {t/2}$ is not a symplectic $(t/2)$-core, then by \cref{thm:sympfact},
we see that 
   \[
   \sp_{\lambda}(X_{\omega})=0.
   \]
 If $\core \lambda {t/2}$ is a symplectic $(t/2)$-core, then 
again by \cref{thm:sympfact},
 we have 
\[
 \sp_{\lambda}(X_{\omega}) = (-1)^{\epsilon} \sgn({\sigma}_{\lambda}) 
 \usp_{\lambda^{(t/2-1)}}(Z^{t/2})
\uoo_{\lambda^{({t}/{4}-1)}}(Z^{t/2})  \times 
\prod_{i=0}^{{(t-8)/{4}}} 
\rs_{\lambda^{(i)},\lambda^{(t/2-2-i)}}(Z^{t/2}),
\]
where the $(t/2)$-quotient $(\lambda^{(0)},\dots,\lambda^{(t/2-1)})$ is evaluated by considering $\lambda \in \P_{tn+t/2}$ 
and 
\[
\epsilon= \ds -\sum_{i={{t}/{4}}}^{(t-4)/2} 
\binom{n_{i}(\lambda)+1}2 + 
{(2n+1)(n+1)}+(2n+1)r. 
\]
Since $\ell(\lambda) \leq \left(2n+1\right)t/4$ and $\core \lambda {t/2}$ is a symplectic $(t/2)$-core, $\ell(\lambda^{(i)})+\ell(\lambda^{(t/2-2-i)}) \leq 2n+1$ {by \cref{len-res-smp}}. Now using
\eqref{rs-s},
\[
 \rs_{\lambda^{(i)},\lambda^{(t/2-2-i)}}(Z^{t/2})
 =s_{(\lambda^{(i)},-\lambda^{(t/2-2-i)})_{2n+1}}(-X^{t/2},-\X^{t/2},-1),
\]
for $0 \leq i \leq {{(t-8)}/{4}}$. So, 
we have,
\begin{equation*}
    \begin{split}
  \sp_{\lambda}(X_{\omega})  =& (-1)^{\epsilon} \sgn({\sigma}_{\lambda})
  \usp_{\lambda^{(t/2-1)}}(-X^{t/2},-\X^{t/2},-1) 
  \uoo_{\lambda^{({t}/{4}-1)}}(-X^{t/2},-\X^{t/2},-1)\\  
  & \times 
 \prod_{i=0}^{{(t-8)/{4}}} s_{(\lambda^{(i)},-\lambda^{(t/2-2-i)})_{2n+1}}(-X^{t/2},-\X^{t/2},-1).
\end{split}
    \end{equation*}
Now we get the required result using~\cref{lem:relations}.     
      \end{proof}

\subsection{Even orthogonal character factorization}

\begin{lem}
\label{lem:relations-1}
We have $\uoe_\lambda(-X,-\X,-1) =(-1)^{|\lambda|} \oo_{\lambda}(X)$ for 
$\lambda \in \P_n$, and $\uoo^{-}_\lambda(-X,-\X,-1) \allowbreak = \oe_{\lambda}(-X,-1)$ for $\lambda \in \P_{n+1}$.
\end{lem}

\begin{proof}
  From \eqref{def-univo}, we have    
\[
\uoe_\lambda(-X,-\X,-1)
=\det \left(
     h_{\lambda_i-i+j}(-X,-\X,-1)- h_{\lambda_i-i-j}(-X,-\X,-1) 
 \right)_{1 \leq i,j \leq n}.
\]
Since $h_m(-X)=(-1)^m h_m(X)$,
\begin{align*}
\uoe_\lambda(-X,-\X,-1) 
&=\det \left(
     (-1)^{\lambda_i-i+j} h_{\lambda_i-i+j}(X,\X,1)- (-1)^{\lambda_i-i-j} h_{\lambda_i-i-j}(X,\X,1) 
 \right)_{1 \leq i,j \leq n}\\
 &=(-1)^{|\lambda|}\det \left(
      h_{\lambda_i-i+j}(X,\X,1)-  h_{\lambda_i-i-j}(X,\X,1) 
 \right)_{1 \leq i,j \leq n}\\
 &=(-1)^{|\lambda|} \oo_{\lambda}(X).
\end{align*}
So, the first equality holds true. 
Similarly, by \eqref{soneg-1}, 
\begin{equation}
\label{uoo}
\uoo_\lambda^-(-X,-\X,-1) 
=\det \left(
     h_{\lambda_i-i+j}(-X,-\X,-1)- h_{\lambda_i-i-j+1}(-X,-\X,-1)  
 \right)_{1 \leq i,j \leq n}.
\end{equation}
Observe that 
\begin{equation*}
\begin{split}
h_m(-X,-\X,-1,-1)-h_{m-2}(-X,-\X,-1,-1)& 
= 
h_m(-X,-\X,-1,-1)\\
+ h_{m-1}(-X,-\X,-1,-1)-h_{m-1}(-X,&-\X,-1,-1)-h_{m-2}(-X,-\X,-1,-1) \\
& = h_m(-X,-\X,-1)-h_{m-1}(-X,-\X,-1).
\end{split}
\end{equation*}
where the second equality comes from using \cref{prop:compid}.  
Applying it {to} the first column of the determinant in \eqref{uoo}, we get
\begin{equation*}
 \begin{split}
     \uoo_\lambda^-(-X,&-\X,-1) 
= \\
\det & \left(
\begin{array}{c|c}
       h_{\lambda_i-i+1}(W,-1)- h_{\lambda_i-i-1}(W,-1)  
   &     \left(  h_{\lambda_i-i+j}(W)- h_{\lambda_i-i-j+1}(W) \right)_{2 \leq j \leq n}  
\end{array} 
 \right)_{1 \leq i \leq n},
 \end{split}   
\end{equation*}
where $W=(-X,-\X,-1)$. 
{On the other end, we have 
\[
\oe_{\lambda}(-X,-1)= \det \left(  h_{\lambda_i-i+j}(W, -1)- h_{\lambda_i-i-j}(W, -1) \right).
\]
In this determinant, applying elementary column operations $C_{j} \rightarrow C_j+C_{j-1}$ for $j=n,\dots,2$ in order and then using \cref{prop:compid} 
completes the proof.}
\end{proof}

We consider the even orthogonal character with the same specialization below. First we recall the following result. 

\begin{lem}[{\cite[Corollary 3.8]{ayyer-kumari-2022}}]
\label{res-len-even}
    Let $\lambda$ be a partition of length 
at most $tn$.
Then core$_t(\lambda)$ is an orthogonal $t$-core if and only if $n_{0}(\lambda)=n \text{ and } n_{i}(\lambda)+n_{t-i}(\lambda)=2n \text{ for } 1 \leq i \leq \floor{\frac{t}{2}}$.
\end{lem}

\begin{thm} 
\label{thm:even}
Let $\lambda$ be a partition in $\P_{tn/2+t/4}$.
The $\OE(tn+{t}/{2})$-character $\oe_{\lambda}(X_{\omega})$ is 
nonzero if and only if $\core \lambda {t/2}$ is orthogonal.
In that case, it is given by 
    \begin{equation}
    \label{even-res}
    \begin{split}
  \oe_{\lambda}(X_{\omega})  = (-1)^{\epsilon} \sgn({\sigma}_{\lambda})  \oo_{\lambda^{(0)}}(X^{t/2}) \;&
\oe_{\lambda^{({t}/{4})}}(-X^{t/2},-1) \\
&\times 
 \prod_{i=1}^{{(t-4)/{4}}} 
 s_{(\lambda^{(i)},-\lambda^{(t/2-i)})_{2n+1}}(X^{t/2},\X^{t/2},1),
 \end{split}
    \end{equation}
    where 
    the $(t/2)$-quotient $(\lambda^{(0)},\dots,\lambda^{(t/2-1)})$ is evaluated by considering $\lambda \in \P_{tn+t/2}$ and 
\[
\epsilon=-\sum_{i=\floor{{(t+4)}/{4}}}^{t/2-1} 
\binom{n_{i}(\lambda)+1}{2} +
{n+1+ |\lambda^{(0)}|}+
+\sum_{i=1}^{{(t-4)/{4}}} 
 |{(\lambda^{(i)},-\lambda^{(t/2-i)})_{2n+1}}|.
\]
\end{thm}

{Notice that the even orthogonal factor in \eqref{even-res} is not a true character.}

\begin{proof} 
Since
$(X_{\omega},\X_{\omega}) = (Z,\omega^2 Z,\dots, \omega^{t-2} Z)$, we have from \eqref{uni-class},
    \[
    \oe_{\lambda}(X_{\omega})=\uoe_\lambda(Z,\omega^2 Z,\dots, \omega^{t-2} Z),
    \]
where  $Z=(\omega X, \omega \X,\omega)$.
 If $\core \lambda {t/2}$ is not an orthogonal $(t/2)$-core, then by \cref{thm:eorthfact},
  we have  
   \[
   \oe_{\lambda}(X_{\omega})=0.
   \]
 If $\core \lambda {t/2}$ is an orthogonal $(t/2)$-core, then
 again by \cref{thm:eorthfact},
 we get
\[
 \oe_{\lambda}(X_{\omega}) = (-1)^{\epsilon} \sgn({\sigma}_{\lambda}) \uoe_{\lambda^{(0)}}(Z^{t/2}) \times
\uoo^{-}_{\lambda^{({t}/{4})}}(Z^{t/2})  \times 
\prod_{i=1}^{{(t-4)/4}} \rs_{\lambda^{(i)},\lambda^{(t/2-i)}}(Z^{t/2}),
\]
where 
\[
\epsilon=-\sum_{i=\floor{{(t+4)}/{4}}}^{t/2-1} 
\binom{n_{i}(\lambda)+1}{2} +
{\frac{(2n+1)(2n+2)}{2}+(2n+2)r},
\]
and 
the $(t/2)$-quotient $(\lambda^{(0)},\dots,\lambda^{(t/2-1)})$ is evaluated by considering $\lambda  \in \P_{tn+t/2}$. 
Since $\ell(\lambda) \leq \left(2n+1\right)t/4$, $\core \lambda {t/2}$ is an orthogonal $(t/2)$-core, $\ell(\lambda^{(i)})+\ell(\lambda^{(t/2-i)}) \leq 2n+1$
by \cref{res-len-even}. Now using \eqref{rs-s}, 
\[
\rs_{\lambda^{(i)},\lambda^{(t/2-i)}}(Z^{t/2})=
s_{(\lambda^{(i)},-\lambda^{(t/2-i)})_{2n+1}}(-X^{t/2},-\X^{t/2},-1),
\] 
for $0 \leq i \leq {{(t-8)}/{4}}$. So, 
we have,
\begin{equation*}
    \begin{split}
  \oe_{\lambda}(X_{\omega})  =& (-1)^{\epsilon} \sgn({\sigma}_{\lambda}) \uoe_{\lambda^{(0)}}(-X^{t/2},-\X^{t/2},-1) 
  \times \uoo^{-}_{\lambda^{({t/4})}}(-X^{t/2},-\X^{t/2},-1) \\
&  \times
 \prod_{i=1}^{{(t-4)/4}} 
 s_{(\lambda^{(i)},-\lambda^{(t/2-i)})_{2n+1}}(-X^{t/2},-\X^{t/2},-1).
\end{split}
    \end{equation*}
Now we get the required result using~\cref{lem:relations-1}.     
\end{proof}

\subsection{Odd orthogonal character factorization}

We now consider the odd orthogonal character with the same specialization. Recall the following result.

\begin{lem}[{\cite[Corollary 3.3]{ayyer-kumari-2022}}]
\label{res-len-odd}
    For a partition $\lambda$ of length at most $tn$,
$\core \lambda t$ is self-conjugate if and only if
\[
n_{i}(\lambda)+n_{t-1-i}(\lambda)=2n, 
\quad 0 \leq i \leq \floor{{(t-1)}/{2}}.
\]
\end{lem}

\begin{thm} 
\label{thm:odd}
Let $\lambda$ be a partition in $\P_{tn/2+t/4}$.
Then the $\OO(tn+{t}/{2}+1)$-character $\oo_{\lambda}(X_{\omega})$ is 
nonzero if and only if $\core \lambda {t/2}$ is self-conjugate.
In that case, it is given by 
    \begin{equation}
    \label{odd-res}
  \oo_{\lambda}(X_{\omega})  = (-1)^{\epsilon} \sgn({\sigma}_{\lambda})  
 \prod_{i=0}^{{(t-4)/4}}
 s_{(\lambda^{(i)},-\lambda^{(t/2-1-i)})_{2n+1}}(X^{t/2},\X^{t/2},1),
    \end{equation}
    where 
    the $(t/2)$-quotient $(\lambda^{(0)},\dots,\lambda^{(t/2-1)})$ is evaluated by considering $\lambda  \in \P_{tn+t/2}$ and 
    \[
    \epsilon=-\ds \sum_{i=\floor{{(t+2)}/{4}}}^{t/2-1} \binom{n_{i}(\lambda)+1}{2}
    +
    \sum_{i=0}^{{(t-4)/4}}
 |{(\lambda^{(i)},-\lambda^{(t/2-1-i)})_{2n+1}}|. 
\]
\end{thm}

\begin{proof} 
Arguing as in the symplectic case,
    \[
    \oo_{\lambda}(X_{\omega})=
 \uoo_\lambda(Z,\omega^2 Z,\dots, \omega^{t-2} Z),   
    \]
    where $Z=(\omega X, \omega \X,\omega)$.
 If $\core \lambda {t/2}$ is not a self-conjugate $(t/2)$-core, then again by
 \cref{thm:oorthfact},
 we get  
   \[
    \oo_{\lambda}(X_{\omega}) =0.
   \]
 If $\core \lambda {t/2}$ is a self-conjugate $(t/2)$-core, then by \cref{thm:oorthfact}, 
 we see that
 \[
 \oo_{\lambda}(X_{\omega}) = (-1)^{\epsilon} \sgn({\sigma}_{\lambda})  \, 
 \prod_{i=0}^{{(t-4)/4}} \rs_{\lambda^{(i)},\lambda^{(t/2-1-i)}}(Z^{t/2}),
\]
where 
\[
\epsilon=-\ds \sum_{i=\floor{{(t+2)}/{4}}}^{t/2-1} \binom{n_{i}(\lambda)+1}{2}, 
\]
and the $(t/2)$-quotient $(\lambda^{(0)},\dots,\lambda^{(t/2-1)})$ is evaluated by considering
$\lambda  \in \P_{tn+t/2}$.
Since $\ell(\lambda) \leq \left(2n+1\right)t/4$, $\core \lambda {t/2}$ is a self-conjugate $(t/2)$-core, $\ell(\lambda^{(i)})+\ell(\lambda^{(t/2-1-i)}) \leq 2n+1$ by \cref{res-len-odd}. Now by \eqref{rs-s}, 
\[
\rs_{\lambda^{(i)},\lambda^{(t/2-1-i)}}(Z^{t/2})
=s_{(\lambda^{(i)},-\lambda^{(t/2-1-i)})_{2n+1}}(-X^{t/2},-\X^{t/2},-1),
\]
for $0 \leq i \leq {{(t-4)}/{4}}$. So, we have  
\begin{equation*}
  \oo_{\lambda}(X_{\omega}) = (-1)^{\epsilon} \sgn({\sigma}_{\lambda})  
 \prod_{i=0}^{{(t-4)/4}} 
 s_{(\lambda^{(i)},-\lambda^{(t/2-1-i)})_{2n+1}}(-X^{t/2},-\X^{t/2},-1). 
\end{equation*}
    This completes the proof.
\end{proof}

\begin{rem}
We note that  
\cref{thm:symp,thm:even,thm:odd} 
for $t=4$ can also be obtained directly from~\cite[Theorem 2.8, Theorem 2.11 and Theorem 2.5]{kumari-2024} respectively.  
\end{rem}  

\section{Relations among characters of classical groups}
\label{sec:rels}

In this section, we look at unexpected equalities between characters of different classical groups. This section is motivated by conjectures by Prasad and Wagh~\cite{prasad-wagh-2020}, which relate irreducible representations of one classical group invariant under a certain automorphism to the twisted characters of the other group.
{To be more precise, they consider
the relationships between the selfdual representations of $\SL(2n)$ and $\text{Spin}(2n+1)$ in \cite[Conjecture 1]{prasad-wagh-2020},
and of $\SL(2n+1)$ and $\Sp(2n)$ in \cite[Conjecture 3]{prasad-wagh-2020}.
Prasad\footnote{private communication} has also considered the 
relationships between the selfdual representations of $\Sp(2n)$ and $\text{Spin}(2n+2)$, as well as $E_6$ and $F_4$.
We thus, compare the characters of the first three pairs of representations in this section.
}
In our notation, their conjectures correspond to {$t = 1$}. We look at more general $t$ and observe new identities for characters.
We will need the following result.

\begin{lem}[{\cite[Lemma 3.15]{ayyer-kumari-2022}}] 
\label{lem:eqqq}
Let $\mu$ be a partition in $\P_{tn}$. 
If $p,q \in \{0,1,\dots,t-1\}$ such that 
$n_{p}(\mu)+n_{q}(\mu)=2n$, then 
\[ 
s_{(\mu^{(q)},-\mu^{(p)})_{2n}}(X,{\X})=s_{(\mu^{(p)},-\mu^{(q)})_{2n}}(X,{\X}).
\]
\end{lem}

\subsection{\texorpdfstring{$\Sp({2tn})$}{TEXT} and \texorpdfstring{$\GL({2tn+1})$}{TEXT}}
\label{sec:sp-gl}

The first pair is the symplectic group $\Sp({2tn})$ and the general linear group 
$\GL({2tn+1})$.
Recall that $\omega$ is the primitive $t$'th root of unity, $X=(x_1,x_2,\dots,x_n)$ and  
let $\tX=(X,\X)$.
Let $\mu \in \P_{tn}$ 
index a representation of $\Sp({2tn})$, and construct
$\lambda=(\pm \mu)_{2tn+1}$.
We recall the following result.

\begin{lem}[{\cite[Lemma 4.2]{ayyer-kumari-2022}}]
\label{sp-gl-nonzero}
With $\mu$ and $\lambda$ as above, $s_{\lambda}(\widetilde{X},\omega \tX,\dots,\omega^{t-1} \tX,1)$ is nonzero if and only if $\sp_{\mu}(X,\omega X, \dots ,\omega^{t-1}X)$ is nonzero.
\end{lem}

This was proved by showing that $\ell{(\core \lambda t)}$ is at most one if and only if $\core \mu t$ is symplectic. In this case, $\core \lambda t=(p)$, where $p \equiv \mu_1 \pmod{t}$ and $p \in \{0, \dots, t-1\}$.

\begin{lem}
\label{lem:quo}
Let $\mu \in \P_{tn}$ with $\core \mu t$ symplectic, and $\lambda=(\pm \mu)_{2tn+1}$ with $\core \lambda t=(p)$ and $p \equiv \mu_1 \pmod{t}$.
Then, for $0 \leq i \leq t-1$, we have 
\[
    \lambda^{(i)}=
    \begin{cases}
\left(\mu^{(i-p-1)}, -\mu^{(p-i-1)}\right)_{2n} & i \neq p,\\
\left( \pm \mu^{(t-1)}\right)_{2n+1} & i = p.
    \end{cases}
    \]
where {the index $k$ in the quotient $\mu^{(k)}$ must be interpreted modulo $t$.}
\end{lem}

\begin{proof} We have 
    \begin{multline*}
\beta(\lambda)=(2\mu_1+2tn,\mu_1+\mu_2+2tn-1,\dots,\mu_1+\mu_{tn}+tn-1,{\mu_{1}+tn}, \\
\mu_1-\mu_{tn}+tn-1,\dots,\mu_1-\mu_2+1,0),     
\end{multline*}
where the first $tn$ parts are of the form $\mu_1+tn+1+\beta(\mu)$ and the last $tn$ parts are of the form $\mu_1+tn-1-\rev(\beta(\mu))$.  
Let 
{$\beta^{(i)}(\mu)$}
be the tuple with coordinates the parts of $\beta(\mu)$ congruent to $i$ modulo $t$ arranged in decreasing order. Then 
\[
\beta^{(i)}(\lambda)=
\begin{cases}
\left(\mu_1+tn+1+\beta^{(i-p-1)}(\mu),\, \mu_1+tn-1-\rev(\beta^{(p-i-1)}(\mu))\right) & i \neq p,\\
\left(\mu_1+tn+1+\beta^{(t-1)}(\mu),\, {\mu_1+tn}, \,\mu_1+tn-1-\rev(\beta^{(t-1)}(\mu))\right) & i = p.
\end{cases}
\]
Therefore, the result holds true.
\end{proof}

We now consider the nonzero characters and prove the following factorization result, which can also be directly obtained from \cref{thm:facx1} by replacing $X$ with the variables $X,\omega X, \dots ,\omega^{t-1}X$.

\begin{thm}
\label{thm:sp-gl}
Suppose $\mu \in \P_{tn}$ and $\lambda=(\pm \mu)_{2tn+1}$. 
Then
\[
s_{\lambda}(\tX,\omega \tX,\dots,\omega^{t-1} \tX,1) 
=  
\sp_{\mu}(X,\omega X, \dots ,\omega^{t-1}X) \;
\oe_{\mu}(X,\omega X, \dots ,\omega^{t-1}X,1).
\]

\end{thm}

\begin{proof}
Suppose $\core \mu t$ is not symplectic. In that case, both sides are zero by \cref{sp-gl-nonzero}.

Now suppose $\core \mu t$ is symplectic. Now 
using \cref{thm:schur-k} with $\tX$ replacing $X$, $m = 1$ and $y = 1$, we see that 
$s_{\lambda}(Y,\omega Y,\dots, \allowbreak \omega^{t-1}Y,1)$ is nonzero since {$\core \lambda t$} has at most one part with value at most $t-1$. Using \cref{lem:quo} and \cref{lem:eqqq}, 
we then obtain 
\begin{equation}
\label{schur-1}
\begin{split}
s_{\lambda}(\tX,\omega \tX,\dots,\omega^{t-1} \tX,1) = \sgn(\sigma_{\lambda}^{\core \lambda t})
    &   \sgn(\sigma_{\emptyset}^{\core \lambda t}) \,  s_{(\pm \, \mu^{(t-1)})_{2n+1}}(\tX^t,1)  \\
\times  \prod_{i=0}^{\floor{\frac{t-3}{2}}} 
\left( s_{\pi_i}(X^t,{\X}^t) \right)^2 
\times 
& \begin{cases}
1 &   t \text{ odd,}\\
s_{(\pm \, \mu^{(t/2-1)})_{2n}}(X^t,{\X}^t) 
&   t \text{ even,}
\end{cases}    
\end{split}
\end{equation}
where $\ds \pi_i=  
\left(\mu^{(i)}, -\mu^{(t-2-i)}\right)_{2n}$ 
for 
$0 \leq i \leq \floor{\frac{t-3}{2}}$.
We further factorize  
the Schur polynomials in the right hand side. 
Using \cref{thm:facx} for the self-dual representation of $\GL({2n})$ indexed by $\left(\pm \, \mu^{(t/2-1)}\right)_{2n}$, we have 
\[
s_{(\pm \, \mu^{(t/2-1)})_{2n}}(X^t,{\X}^t) = (-1)^{\sum_{i=1}^n \mu_i^{(t/2-1)}} \, \oo_{\mu^{(t/2-1)}}(X^t) \, \oo_{\mu^{(t/2-1)}}(-X^t), 
\]
and \cref{thm:facx1} for the self-dual representation of $\GL({2n+1})$ indexed by $\left(\pm \, \mu^{(t-1)}\right)_{2n+1}$, we have 
\[
s_{(\pm \, \mu^{(t-1)})_{2n+1}}(\tX^t,1)=  \sp_{\mu^{(t-1)}}(X^t) \oe_{\mu^{(t-1)}}(X^t,1).
\]
We therefore obtain, using \eqref{schur-1},
\begin{multline}
\label{schur-2-new}
 s_{\lambda}(\tX,\omega \tX,\dots,\omega^{t-1} \tX,1) =\sgn(\sigma_{\lambda}^{\core \lambda t})
      \sgn(\sigma_{\emptyset}^{\core \lambda t}) \, \sp_{\mu^{(t-1)}}(X^t) \, \oe_{\mu^{(t-1)}}(X^t,1) \\
 \prod_{i=0}^{\floor{{(t-3)}/{2}}} 
 \left( s_{{\pi_i}}(X^t,{\X}^t) \right)^2  
 \times 
 \begin{cases}
     1 & t \text{ odd,}\\
\oo_{\mu^{(t/2-1)}}(X^t) \, \oo_{\mu^{(t/2-1)}}(-X^t) &   t \text{ even.}
 \end{cases}      
\end{multline}

Now, 
use \cref{thm:sympfact} and \cref{thm:evenfac-1} to complete the proof. 
\end{proof}

This provides evidence to suggest that a further generalization of \cite[Conjecture 3]{prasad-wagh-2020} might be true.

\subsection{\texorpdfstring{$\OO({2tn+1})$}{TEXT} and \texorpdfstring{$\GL({2tn})$}{TEXT}}
\label{sec:oo-gl}

Along similar lines as in the previous section, we let $\mu \in \P_{tn}$ be a partition 
indexing a representation of $\OO({2tn+1})$, and construct
$\lambda=(\pm \mu)_{2tn}$.
Recall that $\core \lambda t$ is empty if and only if $n_i(\lambda)=2n$ for all $i$~\cite[{Corollary 3.12}]{ayyer-kumari-2022}.

\begin{lem}[{\cite[Corollary 3.3]{ayyer-kumari-2022}}] 
\label{cor:sel-parts}
The $\core \mu t$ is self-conjugate if and only if
\[
n_i(\mu)+n_{t-1-i}(\mu)=2n, \,\, 0 \leq i \leq t-1.
\]
\end{lem}

\begin{lem}
With $\mu$ and $\lambda$ as above, 
$s_{\lambda}(\tX,\omega \tX,\dots,\omega^{t-1} \tX)$ is nonzero if and only if\\
$\oo_{\mu}(X,\omega X, \dots ,\omega^{t-1}X)$ is nonzero.
\end{lem}

\begin{proof}
We prove the result by showing that $\core \lambda t$ is empty if and only if $\core \mu t$ is self-conjugate. We have 
\begin{multline*}
\beta(\lambda)=(2\mu_1+2tn-1,\mu_1+\mu_2+2tn-2,\dots,\mu_1+\mu_{tn}+tn,\\
\mu_1-\mu_{tn}+tn-1,\dots,\mu_1-\mu_2+1,0),     
\end{multline*}
where the first $tn$ parts are of the form $\mu_1+tn+\beta(\mu)$ and the last $tn$ parts are of the form $\mu_1+tn-1-\rev(\beta(\mu))$. 
Suppose $p \equiv \mu_1 \pmod{t}$ for some $p \in 
\{0, \dots, t-1\}$. 
Then 
\[
\beta(\lambda)_i+\beta(\lambda)_{2tn+1-i} \equiv 2p-1 \pmod{t}, \,\, i \in \{1,2,\dots,tn\}.
\]
So for all $j \in \{0, \dots, t-1\} $, the parts of $\beta(\lambda)$ less than $\mu_1+tn$ congruent to $j$ are in bijective correspondence with parts at least $\mu_1+tn$ congruent to $2p-1-j$.
Hence the number of parts of $\beta(\lambda)$ at least $\mu_1+tn$ congruent to $j$ and $2p-1-j$ is $n_j(\lambda)$. 

If $\zeta \equiv i \pmod{t}$ occurs in $\beta(\mu)$, then $\mu_1+tn+\zeta \equiv p+i \pmod{t}$, which is at least $\mu_1+tn$, occurs in $\beta(\lambda)$. 
For $j=i_p \coloneqq (p+i) \pmod{t}$, we see that 
\begin{equation}
\label{nmu1}
 n_i(\mu)+n_{t-1-i}(\mu)= n_{i_p}(\lambda), \quad i \in \{0, \dots, t-1\} .  
\end{equation}

First, suppose $\core \lambda t$ is empty. Then $n_i(\lambda)=2n$ for $i \in \{0, \dots, t-1\} $. Substitution of this in \eqref{nmu1} gives, $n_i(\mu)+n_{t-1-i}(\mu)=2n$ for $i \in \{0, \dots, t-1\} $. 
From \cref{cor:sel-parts}, it now follows that $\core \mu t$ is self-conjugate.

Now assume $\core \mu t$ is self-conjugate. Using \cref{cor:sel-parts} in \eqref{nmu1}, we see that $n_i(\lambda)=2n$. This shows that $\core \lambda t$ is empty.
\end{proof}

\begin{lem}
\label{lem:quoeq}
With $\mu$ and $\lambda$ as above and $0 \leq i \leq t-1$, we have 
\[
    \lambda^{(i)}= 
\left(\mu^{(i-p)}, -\mu^{(p-i-1)}\right)_{2n},
    \]
where $p \equiv \mu_1 \pmod{t}$. 
\end{lem}

The proof of \cref{lem:quoeq} proceeds along very similar lines as the proof of \cref{lem:quo}, {and we leave it to the interested reader}.

We now consider the nonzero characters and prove the following factorization result, which can also be directly obtained from \cref{thm:facx} by replacing $X$ with the variables $X,\omega X, \dots ,\omega^{t-1}X$.

\begin{thm}
\label{thm:oo-gl}
Suppose $\mu \in \P_{tn}$ and $\lambda=(\pm \mu)_{2tn}$. Then 
\[
s_{\lambda}(\tX,\omega \tX,\dots,\omega^{t-1} \tX)=(-1)^{|\mu|}
\oo_{\mu}(X,\omega X, \dots ,\omega^{t-1}X) \;
\oo_{\mu}(-X,-\omega X, \dots ,-\omega^{t-1}X).
\]
\end{thm}
\begin{proof}
Using \cref{thm:schur-k} with $\tX$ replacing $X$, $m = 0$, we see that 
$s_{\lambda}(\tX,\omega \tX,\dots,\omega^{t-1} \tX)$ is nonzero since {$\core \lambda t$} is empty. Using \cref{lem:quoeq} and \cref{lem:eqqq}, 
\begin{equation}
\label{schur-2}
\begin{split}
s_{\lambda}(\tX,\omega \tX,\dots,\omega^{t-1} \tX)=
\sgn(\sigma_{\lambda})  \sgn(\sigma_{\emptyset})\, &
\prod_{i=0}^{\floor{{(t-2)}/{2}}} \left( s_{(\mu^{(i)},-\mu^{(t-1-i)})_{2n}}(X^t,{\X}^t)\right)^2 \\
\times &
\begin{cases}
    s_{  (\pm \, \mu^{({(t-1)}/{2})})_{2n}}(X^t,\X^t) & t \text{ odd,}\\
    1 & t \text{ even.}
\end{cases}
\end{split}
\end{equation}
We can further factorize  
the Schur polynomials on the right hand side
in the following way. Using \cref{thm:facx}, 
we have
\[
s_{(\pm \, \mu^{({(t-1)}/{2})})_{2n}}(X^t,\X^t)
= (-1)^{|{\mu^{((t-1)/2)}}|} \oo_{\mu^{({(t-1)}/{2})}}(X^t) \oo_{\mu^{({(t-1)}/{2})}}(-X^t). 
\]
So, by \eqref{schur-2}, we get 
\begin{equation*}
\begin{split}
s_{\lambda}(\tX,\omega \tX,\dots,\omega^{t-1} \tX)=
\sgn(\sigma_{\lambda}) & \sgn(\sigma_{\emptyset})  
\prod_{i=0}^{\floor{{(t-2)}/{2}}} \left( s_{(\mu^{(i)},-\mu^{(t-1-i)})_{2n}}(X^t,{\X}^t)\right)^2 \\
\times &
\begin{cases}
 {(-1)^{|{\mu^{((t-1)/2)}}|}} \oo_{\mu^{({(t-1)}/{2})}}(X^t) 
\oo_{\mu^{({(t-1)}/{2})}}(-X^t)  & t \text{ odd,}\\
    1 & t \text{ even.}
\end{cases}
\end{split}
\end{equation*}
Now, using \cref{thm:oorthfact}, we have that 
\begin{equation*}
    \begin{split}
        \oo_{\mu}(X,\omega X,& \dots ,\omega^{t-1}X) \;
\oo_{\mu}(-X,-\omega X, \dots ,-\omega^{t-1}X) =  \\
(-1)^{\epsilon} & \prod_{i=0}^{\floor{{(t-2)}/{2}}} \left( s_{(\mu^{(i)},-\mu^{(t-1-i)})_{2n}}(X^t,{\X}^t)\right)^2 \times
\begin{cases}
   \oo_{\mu^{({(t-1)}/{2})}}(X^t) 
\oo_{\mu^{({(t-1)}/{2})}}(-X^t)  & t \text{ odd,}\\
    1 & t \text{ even.}
\end{cases}
    \end{split}
\end{equation*}
where 
\[
\epsilon=\sum_{i=0}^{\floor{{(t-2)}/{2}}} (|(\mu^{(i)}|+|\mu^{(t-1-i)}|).
\]
This completes the proof. 
\end{proof}

This provides evidence to suggest that a further generalization of \cite[Conjecture 1]{prasad-wagh-2020} might be true.

\subsection{\texorpdfstring{$\OE({2tn+2})$}{TEXT} and \texorpdfstring{$\Sp({2tn})$}{TEXT}}
\label{sec:oe-sp}

The third pair is the symplectic group $\Sp({2tn})$ and the even orthogonal group 
$\OE({2tn+2})$.
Suppose $\mu \in \P_{tn}$ is a partition 
indexing a representation of $\Sp({2tn})$, and construct
$\lambda=\mu_{tn}+(\mu,0)$, which has length at most $tn+1$,
then $\lambda$ indexes a representation of $\OE({2tn+2})$. 

Unfortunately, the story for this pair is not the same as that of the previous subsections.
Suppose $\mu_{tn} \neq 0$ and $\sp_{\mu}(X,\omega X, \dots,\omega^{t-1}X)$ is nonzero. Then that does not imply that
$\oe_{\lambda}(X, \omega X, \dots, \omega^{t-1}X, 1)$ is nonzero. Consider the following example.

\begin{eg}
Assume $t=5$, $n=2$. If   $\mu=(9,8,6,5,4,4,3,2,2,1)=(8,6,3,1|9,7,4,2)$, a symplectic $5$-core of length at most $10$, then $\lambda=\mu_{tn}+(\mu,0)=(10,9,7,6,5,5,4,3,3,2,1)$ of length $11$. 
Here $\sp_{\mu}(X,\omega X, \dots,\omega^{t-1}X)$ is nonzero but $\oe_{\lambda}(X,\omega X, \dots ,\omega^{t-1}X,1)$ is zero.
\end{eg}

But there is one implication that is true.
If $\mu_{tn}=0$, then $\lambda$ is the same as $\mu$ and we have the following lemma.

\begin{lem}
\label{o-sp}
Let $\mu$ be a partition in $\P_{tn}$ such that
$\sp_{\mu}(X,\omega X, \dots,\omega^{t-1}X)$ is nonzero. Then
$\oe_{\mu}(X,\omega X,\dots,\omega^{t-1}X,1)$ is nonzero. 
\end{lem}

\begin{proof}
$\sp_{\mu}(X,\omega X, \dots,\omega^{t-1}X)$ is nonzero if and only if $\core \mu t$ is symplectic.
By \cref{thm:evenfac-1}, 
if $\core \mu t$ is symplectic, then $\oe_{\mu}(X,\omega X,\dots,\omega^{t-1}X,1)$ is nonzero. This completes the proof. 
\end{proof}

\begin{rem}
The converse of \cref{o-sp} is not true. 
We note that $\sp_{(1)}(X,-X) = 0$ but $\oe_{(1)}(X,-X,1)$ is non-zero.
Some more examples of $3$-cores $\mu$
for which $\sp_{\mu}(x_1,\omega x_1,\omega^{2}x_1)$ is zero but $\oe_{\mu}(x_1,\omega x_1,\omega^{2}x_1,1)$ is non-zero are $(1),(2)$ and $(4,2)$, {where $\omega$ is a third root of unity}.
\end{rem}

\begin{thm} Suppose $\mu \in \P_{tn}$  
and $\core \mu t$ is symplectic of rank $r$. Then  
\begin{equation}
\begin{split}
\sp_{\mu}(X,\omega X, \dots ,\omega^{t-1}X) =  (-1)^{\epsilon} \sgn(\sigma_{\mu}) \, \sp_{\mu^{(t-1)}}(X^t) \prod_{i=0}^{{(t-3)}/{2}} & s_{(\mu^{(i)}, -\mu^{(t-2-i)})_{2n}}(X^t,{\X}^t)\\
\times & {\begin{cases}
    \oo_{\mu^{((t-2)/2)}}(X^t) & t \text{ even},\\
    1 & t \text{ odd,}
\end{cases}}
\end{split}
\end{equation}
where 
\[
\epsilon= \ds -\sum_{i=\floor{{t}/{2}}}^{t-2} 
\binom{n_{i}(\mu,tn)+1}2 +
\begin{cases}
\frac{n(n+1)}{2}+nr & t \text{ even},\\
0 & t \text{ odd,}
\end{cases}
\]
and
{\begin{equation}
\begin{split}
\oe_{\mu}(X,\omega X, \dots ,\omega^{t-1}X,1) =  (-1)^{\epsilon'} \sgn(\sigma_{\mu}) \,
 &\oe_{\mu^{(t-1)}}(X^t,1) 
 \prod_{i=0}^{{(t-3)}/{2}}  s_{(\mu^{(i)}, -\mu^{(t-2-i)})_{2n}}(X^t,{\X}^t)\\
\times & {\begin{cases}
   (-1)^{|\mu^{((t-2)/2)}|} \oo_{\mu^{((t-2)/2)}}(-X^t) & t \text{ even},\\
    1 & t \text{ odd,}
\end{cases}}
\end{split}
\end{equation}}
where 
\[
\epsilon' =
\ds \sum_{i=\floor{{t}/{2}}}^{t-2}  \Bigg( \binom{n_i(\mu,tn)}{2}+ (t-1)n(n_i(\mu,tn)-n) \Bigg).
\]
\end{thm}

\begin{proof}
The first factorization here uses \cref{thm:sympfact}. The second comes from \cref{thm:evenfac-1}, where we consider $\mu \in \P_{tn}$. Note that $n_{i}(\mu,tn)=n_{i+1}(\mu,tn+1)$, $0 \leq i \leq t-2$, $n_{t-1}(\mu,tn)=n_{0}(\mu,tn+1)-1$.
\end{proof}

The results in this section suggest that a further generalization of \cite[Conjecture 4]{prasad-wagh-2020} might be false.

\section{Independence of twisted characters of variables}
\label{sec:staircase}

Recall that $t \geq 2$ and $\omega$ is a primitive $t$'th root of unity. Suppose $X=(x_1,\dots,x_n)$ as before, and $Y=(y_1,\dots,y_m)$. Recall the definition of skew Schur polynomial given in \eqref{def:SSkew-Jcob}.
Then by~\cite[Chapter I.5, Equation (5.9)]{macdonald-2015},  
\begin{equation}
\label{split-schur}
s_{\lambda}(X,Y) = \sum_{\mu \subset \lambda} s_{{\mu}}(X) s_{\lambda/\mu}(Y).
\end{equation}
By~\cite[Equation 2.1]{koike-terada-1990}, we have  
\begin{equation}
\label{split-ortho}
    \uoe_{\lambda}(X,Y) = \sum_{\mu \subset \lambda} \uoe_{\mu}(X) s_{\lambda/\mu}(Y).
\end{equation}
Lastly by~\cite[Page 117]{koike-terada-1990}, we have  
\begin{equation}
\label{split-symp}
    {\usp}_{\lambda}(X,Y) = \sum_{\mu \subset \lambda} {\usp}_{\mu}(X) s_{\lambda/\mu}(Y).
\end{equation}

We first recall the factorization of the skew Schur polynomial twisted by roots of unity.

\begin{lem}
\label{lem:non-zero}
    Let $\lambda$ and $\mu$ be partitions in $\P_{tm}$.
    Then the specialized skew Schur polynomial $s_{\lambda/\mu}(Y,\omega Y,\dots, \omega^{t-1} Y)$ is given as follows.
\begin{enumerate}
    \item If $\core \lambda t \neq \core \mu t$,
    then 
    \[
    s_{\lambda/\mu}(Y,\omega Y,\dots, \omega^{t-1} Y)=0.
    \]
    \item If $\core \lambda t = \core \mu t$,
    then
    \begin{equation*}
    s_{\lambda/\mu}(Y,\omega Y,\dots, \omega^{t-1} Y)
        = \sgn(\sigma_{\lambda})
        \sgn(\sigma_{\mu})
        \prod_{i=0}^{t-1} s_{\lambda^{(i)}/\mu^{(i)}}(Y^t).
    \end{equation*}
\end{enumerate} 
\end{lem}

\begin{proof}
    By the Jacobi-Trudi type identity \eqref{def:SSkew-Jcob}, we see that the required skew Schur polynomial is
\begin{equation}
   s_{\lambda/\mu}(Y,\omega Y,\dots, \omega^{t-1} Y) =  
   \det(h_{\beta_i(\lambda)-\beta_j(\mu)}(Y,\omega Y,\dots, \omega^{t-1} Y)). 
\end{equation}
Permuting the rows and columns of the determinant by $\sigma_{\lambda}$ and $\sigma_{\mu}$ respectively, defined in \eqref{sigma-perm}, we see that the skew Schur polynomial is 
\[
\sgn(\sigma_{\lambda}) \sgn(\sigma_{\mu}) \det \left( h_{\beta_{\sigma_{\lambda}(i)}(\lambda)-\beta_{\sigma_\mu(j)}(\mu)}(Y,\omega Y,\dots, \omega^{t-1} Y)\right),
\]
By \cref{thm:schur-fac} for $\lambda=(j)$, we have $h_j(Y,\omega Y,\dots, \omega^{t-1} Y)=0$ if $j \not \equiv 0 \pmod{t}$ and $j>0$. 
Also, note that $h_j(Y,\omega Y,\dots, \omega^{t-1} Y)=0$ if $j<0$. 
Substituting these values in the above determinant, we see that the skew Schur polynomial is
\[
\sgn(\sigma_{\lambda}) \sgn(\sigma_{\mu})
\det \left(
\begin{array}{ccccc}
\widehat{S}_0 &&& \\
 & \widehat{S}_1  & & \text{\huge0}\\
 && \ddots \\
  \text{\huge0} &    &   & \widehat{S}_{t-1}
\end{array}
\right), 
\]
where
\[\widehat{S}_p = \left(h_{\beta^{(p)}_{i}(\lambda)-\beta^{(p)}_{j}(\mu)}(Y,\omega Y,\dots, \omega^{t-1} Y)\right)_{\substack{1 \leq i \leq n_{p}(\lambda)\\1 \leq j \leq n_{p}(\mu)}}, \quad p \in [0,t-1].
\]
If $\core \lambda t \neq \core \mu t$, then $n_i(\lambda) \neq n_i(\mu)$ for some $i \in [0,t-1]$ and the $(i+1)$'th diagonal block is not a square block. So,
$
s_{\lambda/\mu}(Y,\omega Y,\dots, \omega^{t-1} Y)
=0.
$
If $\core \lambda t=\core \mu t$, then  $n_i(\lambda) = n_i(\mu)$ for all $0 \leq i \leq t-1$. In this case, $\det(\widehat{S}_p) = s_{\lambda^{(p)}/\mu^{(p)}}(Y^t)$. This gives
 the desired result.
\end{proof}

{Note that $s_{\lambda/\mu}(Y,\omega Y,\dots, \omega^{t-1} Y)$ could be zero even if $\core \lambda t = \core \mu t$, since one of the quotients $\lambda^{(i)}/\mu^{(i)}$ may not be defined.}
A further consequence is the following.

\begin{lem}
\label{lem:iff-m}
Suppose $\lambda \in \P_{tm}$.
Then 
$s_{\lambda/\mu}(Y,\omega Y,\dots, \omega^{t-1} Y) = 0 \text{ for all } \mu \subsetneq \lambda$  if and only if 
$\lambda=\core \lambda t$.
\end{lem}

\begin{proof}
If $\lambda=\core{\lambda}{t}$ and $\mu \subsetneq \lambda$, then    $\core{\lambda}{t} \neq \core{\mu}{t}$. So, by \cref{lem:non-zero},
\[
s_{\lambda/\mu}(Y,\omega Y,\dots, \omega^{t-1} Y) = 0.
\]
    Suppose $s_{\lambda/\mu}(Y,\omega Y,\dots, \omega^{t-1} Y) = 0 \text{ for all } \mu \subsetneq \lambda$. If possible, assume $\lambda \neq \core{\lambda}{t}$. Then
    \[
    s_{\lambda/ \core{\lambda}{t}}(Y,\omega Y,\dots, \omega^{t-1} Y) = 0,
    \]
    which contradicts \cref{lem:non-zero}. So $\lambda = \core{\lambda}{t}$ and the result
holds.   
\end{proof}

We now conclude the result on the independence of universal characters of variables twisted by roots of unity.

\begin{thm} 
\label{thm:universal-m} 
Let $\lambda$ be a partition in $\P_{n-tm}$.
Suppose $f_{\lambda}(X) \in \{s_{\lambda}(X), \usp_{\lambda}(X), \uoe_{\lambda}(X)\}$ is a universal character. 
Then
    \begin{equation}
    \label{schureq-m}
        f_{\lambda}(x_1,\dots,x_{n-tm},Y,\omega Y,\dots, \omega^{t-1} Y) = f_{\lambda}(x_1,\dots,x_{n-tm})
    \end{equation}
    if and only if $\lambda=\core{\lambda}{t}$.    
\end{thm}

\begin{proof} By \eqref{split-schur}, \eqref{split-ortho} and \eqref{split-symp}, we have
    \[
    f_{\lambda}(x_1,\dots,x_{n-tm},Y,\omega Y,\dots, \omega^{t-1} Y) = 
    \sum_{\mu \subset \lambda} f_{\mu}(x_1,\dots,x_{n-tm}) 
    s_{\lambda/\mu}(Y,\omega Y,\dots, \omega^{t-1} Y).
    \]
    Then \eqref{schureq-m} holds if and only if 
    \begin{equation}
    \label{vanish-m}
     \sum_{\mu \subsetneq
    \lambda} f_{\mu}(x_1,\dots,x_{n-tm}) 
    s_{\lambda/\mu}(Y,\omega Y,\dots, \omega^{t-1} Y) = 0.   
    \end{equation}
    Since the set of universal characters $\{f_{\lambda}(X) \mid \lambda \in \mathcal{P} \}$ is a linearly independent set, \eqref{vanish-m} holds if and only if
    \begin{equation}
    \label{vanish-1-m}
    s_{\lambda/\mu}(Y,\omega Y,\dots, \omega^{t-1} Y) = 0 \text{ for all } \mu \subsetneq \lambda.    
    \end{equation}
By \cref{lem:iff-m}, \eqref{vanish-1-m} holds if and only if $\lambda=\core \lambda t$. This completes the proof.    
\end{proof}

\begin{rem}
      We note that \cref{thm:universal-m} does not hold if $\ell(\lambda)>n-tm$. For instance, suppose $t = 2, m = 1, n = 3$ and 
$\lambda=(4,1)$. Then $\core \lambda 2  = (2, 1) \neq  \lambda$, but $s_{(4,1)}(x,y,-y)=0=s_{(4,1)}(x_1)$. However, 
      if $\lambda$ is a partition such that $\ell(\lambda)>n-tm$ and $\lambda = \core \lambda t$, then \eqref{schureq-m} holds. 
\end{rem}

\begin{cor}
 Let $\lambda$ be a partition in $\P_{n-tm}$.
 Suppose $C_{\lambda}(X) \in \{\sp_{\lambda}(X), \oo_{\lambda}(X), \oe_{\lambda}(X)\}$ is a classical character and $\lambda=\core{\lambda}{t}$. 
 Then
    \begin{equation}
    \label{schureq-classical}
        C_{\lambda}(x_1,\dots,x_{n-tm},Y,\omega Y,\dots, \omega^{t-1} Y) = C_{\lambda}(x_1,\dots,x_{n-tm}).
    \end{equation} 
\end{cor}
\begin{cor} For $n \geq tm+1$, we have 
    \begin{equation}
    \label{complete}
        h_r(x_1,\dots,x_{n-tm},Y,\omega Y,\dots, \omega^{t-1} Y) = h_r(x_1,\dots,x_{n-tm})
    \end{equation} 
    if and only if $0 \leq r \leq t-1$.
\end{cor}

\subsection{Equality of hook Schur and Schur polynomials}

The hook Schur polynomial $\hs_{\lambda}(X/Y)$ is a character of an irreducible covariant tensor representation
of $\gl(m/n)$ introduced by Berele and Regev~\cite{berele1987hook}. It is given by
 \begin{equation}
 \label{hs-formula}
    \hs_{\lambda}(X/Y) = \sum_{\mu} s_{\mu}(X) s_{\lambda'/\mu'}(Y). 
 \end{equation}
One of the fundamental properties of the hook Schur polynomial $\hs_{\lambda}(X/Y)$ is that it gives an expression independent of $u$ when substituting $x_n=u$ and $y_m=-u$~\cite[Example I.3.23]{macdonald-2015}. 
The following result presents a factorization formula for hook Schur polynomials.

\begin{thm}[{\cite[Theorem 6.20]{berele1987hook}}]
\label{thm:hookschurfac}
Suppose $\lambda$ is a partition such that $\lambda_n \geq m \geq \lambda_{n+1}$. Then $\lambda$ can be written in the form $(m+\tau, \eta)$, with $\tau \in \P_{n}$ and $\eta' \in \P_m$. We have 
\[
\hs_{\lambda}(X/Y)= s_{\tau}(X) s_{\eta'}(Y) \prod_{i=1}^n \prod_{j=1}^m  (x_i+y_j).
\]
\end{thm}

{In a separate vein, we have an interesting equality for skew Schur polynomials.} 
We say that ${\delta_n}=(n,n-1,\dots,1)$ is a \emph{staircase partition}. 

\begin{lem}[{\cite[Corollary 7.32]{reiner-shaw-willigenburg-2007}}]
\label{lem:staircase}
For $n \geq 1$, and all $\mu \subset {\delta_n}$,
    \(
    s_{{\delta_n}/\mu}({\uX})=s_{{\delta_n}/\mu'}({\uX}).
    \)
\end{lem}

Using  \cref{lem:staircase} in \eqref{split-schur}, we have the following result.

\begin{cor}
\label{cor:hook-schur}
If $\lambda$ is a staircase partition with $\ell(\lambda) \leq n+m$, then 
\[
   \hs_{\lambda}(X/Y) = s_{\lambda}(X,Y).
    \]
\end{cor}

Using \cref{thm:hookschurfac} and \cref{cor:hook-schur}, we obtain the following  factorization of Schur polynomials of staircase shapes.

    \begin{thm} 
    \label{thm:schur-fact}
    We have 
        \begin{equation}
            s_{{\delta_{n+m}}}(X,Y)=
            s_{{\delta_{n}}}(X) s_{{\delta_{m}}}(Y) \prod_{i=1}^n \prod_{j=1}^m  (x_i+y_j),
        \end{equation}
    and
        \begin{equation}
            s_{{\delta_{n+m-1}}}(X,Y)= s_{{\delta_{n-1}}}(X) s_{{\delta_{m-1}}}(Y)
            \prod_{i=1}^n \prod_{j=1}^m (x_i+y_j).
        \end{equation}
        \end{thm}

{We also obtain a partial converse to \cref{cor:hook-schur}.}

\begin{thm} 
\label{hsequality}
Suppose $\lambda$ is a partition with $\ell(\lambda) \leq n+m-2$. Then
    \[
    \hs_{\lambda}(X/Y)= s_{\lambda}(X,Y)
    \]
if and only if  $\lambda$ is a staircase partition.
\end{thm}

\begin{proof} 
Using \eqref{hs-formula} for a staircase partition $\delta$, and applying
\cref{cor:hook-schur},
we see that
\[
\hs_{\delta}(X/Y) = s_{\delta}(X,Y).
\]
For the converse, suppose $\lambda$ is not a staircase partition. Then by  \cref{thm:universal-m} for $t=2$, $s_{\lambda}(x_1,\dots,x_{n-1},y_1,\dots,y_{m-1},u,-u)$ 
depends on $u$. So, $s_{\lambda}(X,Y)$ depends on $u$, 
but $\hs_{\lambda}(X/Y)$ is independent of $u$ after substituting $x_n=u$ and $y_m=-u$.
So, $\hs_{\lambda}(X/Y) \neq s_{\lambda}(X,Y)$. This proves the result.
\end{proof}

\section{Universal characters at roots of unity}
\label{sec:univ roots of unity}

In this section, we will generalize Littlewood's result on evaluation of Schur polynomials at 
roots of unity in \cref{thm:roots of unity} to other universal characters.
For each of the classical characters, we will show that these evaluations are
$0, 1$ or $2$ up to sign. As before, recall that $t \geq 2$ is fixed and $\omega$ is a primitive $t$'th root of unity.

We will need the following observation.
Suppose $\lambda$ and $\mu$ are partitions such that $\ell(\lambda)+\ell(\mu) \leq 2$. Recall the definition of
$\rs_{\lambda,\mu}$ from \eqref{defrs}. Then we have
\begin{equation}
\label{rseval}
\rs_{\lambda,\mu}(1)=\sum_{\nu} (-1)^{|\nu|}s_{\lambda/\nu}(1) 
s_{\mu/\nu'}(1) = \begin{cases}
    1 & \ell(\lambda)+\ell(\mu) \in \{0,1\},\\
    0 & \ell(\lambda)+\ell(\mu)=2.
\end{cases}
\end{equation}

\begin{thm}
Suppose $\lambda$ is a partition in $\P_t$. 
Then ${\usp}_{\lambda}(1,\omega, \dots ,\omega^{t-1})$ is nonzero if and only if
$\core \lambda t$ is a symplectic $t$-core and $\ell(\lambda^{(i)})+\ell(\lambda^{(t-2-i)}) \leq 1$ for all $0 \leq i \leq \floor{{(t-3)}/{2}}$. 
In that case, if $\core \lambda t$ has rank $r$,
\begin{equation*}
{\usp}_{\lambda}(1,\omega, \dots ,\omega^{t-1}) =  (-1)^{\epsilon} 
\sgn (\sigma_{\lambda})
 \times \begin{cases}
2 & t \text{ even and } \lambda^{(t/2-1)} \neq \emptyset,\\
1 & t \text{ even and } \lambda^{(t/2-1)}=\emptyset,
\\1 & t \text{ odd},
    \end{cases} 
\end{equation*}
where 
\[
\epsilon= \ds -\sum_{i=\floor{{t}/{2}}}^{t-2} \binom{n_{i}(\lambda)+1}2 
+ \begin{cases}
1+r & t \text{ even},
\\0  & t \text{ odd}.
\end{cases}
\]
\end{thm}

\begin{proof}
If $\core \lambda t$ is not a symplectic $t$-core, then by
\cref{thm:sympfact},  
\[
{\usp}_{\lambda}(1,\omega , \dots ,\omega^{t-1}) =  0.
\]
If it is, then again by the same result, we have 
\begin{equation*}
{\usp}_{\lambda}(1,\omega , \dots ,\omega^{t-1}) =  (-1)^{\epsilon} 
\sgn (\sigma_{\lambda}) \, \usp_{\lambda^{(t-1)}}(1) 
 \prod_{i=0}^{\floor{{(t-3)}/{2}}} 
  \rs_{\lambda^{(i)},  \lambda^{(t-2-i)}}(1)
  \times 
\begin{cases}
\uoo_{\lambda^{(t/2-1)}}(1) 
& t \text{ even} ,\\
1 & t \text{ odd},
\end{cases}  
\end{equation*}
where 
\[
\epsilon= \ds -\sum_{i=\floor{{t}/{2}}}^{t-2} \binom{n_{i}(\lambda)+1}2 
+ \begin{cases}
1+r & t \text{ even},
\\0  & t \text{ odd},
\end{cases}
\]
and $\ell(\lambda^{(t-1)}), \,\,
\ell(\lambda^{(t/2-1)}) \leq 1$ and $\ell(\lambda^{(i)})+\ell(\lambda^{(t-2-i)}) \leq 2$
for 
$0 \leq i \leq \floor{{(t-3)}/{2}}$.
By definition,
\[
\usp_{\lambda^{(t-1)}}(1)=h_{\lambda^{(t-1)}}(1) = 1 
\]
and 
\[
\uoo_{\lambda^{(t/2-1)}}(1) = {h_{\lambda^{(t/2-1)}}(1)+h_{\lambda^{(t/2-1)}-1}(1)}=\begin{cases}
    1 & \ell(\lambda^{(t/2-1)}) = 0, \\
    2 & \ell(\lambda^{(t/2-1)}) > 0.
\end{cases}
\]
We now use \eqref{rseval} to conclude the result.
\end{proof}

\begin{thm}
\label{thm:uoe roots of unity}
Suppose $\lambda$ is a partition in $\P_t$.
Then $\uoe_{\lambda}(1,\omega, \dots ,\omega^{t-1})$ is nonzero if and only if $\core \lambda t$ is {an orthogonal} $t$-core, $\ell(\lambda^{(i)})+\ell(\lambda^{(t-i)}) \leq 1$ for all $1 \leq i \leq \floor{{t}/{2}}$ and $\lambda^{(0)} \in \{\emptyset,(1)\}$.
In that case, if $\core \lambda t$ has rank $r$,
\begin{equation*}
\uoe_{\lambda}(1,\omega, \dots ,\omega^{t-1}) =  (-1)^{\epsilon} 
\sgn (\sigma_{\lambda})
 \times \begin{cases}
0 & t \text{ even and }
\lambda^{(t/2)} \neq \emptyset,\\
1 & t \text{ even and } 
\lambda^{(t/2)}=\emptyset,
\\1 & t \text{ odd},
    \end{cases} 
\end{equation*}
where 
\[
\epsilon= \ds -\sum_{i=\floor{(t+2)/{2}}}^{t-1} \binom{n_{i}(\lambda)+1}2 +r 
+ \begin{cases}
1+r & t \text{ even},
\\0  & t \text{ odd}.
\end{cases}
\]
\end{thm}

\begin{proof}
If $\core \lambda t$ is not an orthogonal $t$-core, then by
\cref{thm:eorthfact},
\[
\uoe_{\lambda}(1,\omega , \dots ,\omega^{t-1}) =  0.
\]
If it is, then again by the same result, we have, 
\begin{equation*}
\uoe_{\lambda}(1,\omega , \dots ,\omega^{t-1}) =  (-1)^{\epsilon} 
\sgn (\sigma_{\lambda}) \times \uoe_{\lambda^{(0)}}(1) 
 \times \prod_{i=1}^{\floor{{(t-1)}/{2}}} 
  \rs_{\lambda^{(i)},  \lambda^{(t-i)}}(1)
  \times 
\begin{cases}
\uoo^-_{\lambda^{(t/2)}}(1) 
& t \text{ even} ,\\
1 & t \text{ odd},
\end{cases}  
\end{equation*}
where 
\[
\epsilon= \ds -\sum_{i=\floor{(t+2)/{2}}}^{t-1} \binom{n_{i}(\lambda)+1}2 +r 
+ \begin{cases}
1+r & t \text{ even},
\\0  & t \text{ odd},
\end{cases}
\]
$\ell(\lambda^{(0)}), \,\,
\ell(\lambda^{(t/2)}) \leq 1$ and $\ell(\lambda^{(i)})+\ell(\lambda^{(t-i)}) \leq 2$
for 
$1 \leq i \leq \floor{{(t-1)}/{2}}$. Observe that
\[
\uoe_{\lambda^{(0)}}(1)
=h_{\lambda^{(o)}}(1)-h_{\lambda^{(o)}-2}(1) = \begin{cases}
   1 & \lambda^{(0)} \in \{\emptyset, (1) \}\\
   0 & \text{ otherwise.}
\end{cases}
\]
and 
\[
\uoo^-_{\lambda^{(t/2)}}(1) = h_{\lambda^{(t/2)}}(1)
-h_{\lambda^{(t/2)}-1}(1)
=\begin{cases}
    1 & \lambda^{(t/2)} = \emptyset, \\
    0 & \lambda^{(t/2)} \neq \emptyset.
\end{cases}
\]
Therefore, using \eqref{rseval} completes the proof.
\end{proof}

\begin{thm}
Suppose $\lambda$ is a partition in $\P_t$. 
Then $\uoo_\lambda(1,\omega, \dots ,\omega^{t-1})$ is nonzero if and only if
$\core \lambda t$ is a self-conjugate $t$-core and $\ell(\lambda^{(i)})+\ell(\lambda^{(t-1-i)}) \leq 1$ for all $0 \leq i \leq \floor{{(t-2)}/{2}}$.
In that case, if $\core \lambda t$ has rank $r$, 
\begin{equation*}
\uoo_\lambda(1,\omega, \dots ,\omega^{t-1}) =  (-1)^{\epsilon} 
\sgn (\sigma_{\lambda})
 \times \begin{cases}
1 & t \text{ even},\\
1 & t \text{ odd and } \lambda^{((t-1)/2)}=\emptyset,
\\
2 & t \text{ odd and } \lambda^{((t-1)/2)} \neq \emptyset,
    \end{cases} 
\end{equation*}
where 
\[
\epsilon= \ds -\sum_{i=\floor{{(t+1)}/{2}}}^{t-1} \binom{n_{i}(\lambda)+1}2 
+ \begin{cases}
0 & t \text{ even},
\\r  & t \text{ odd}.
\end{cases}
\]
\end{thm}

\begin{proof}
If $\core \lambda t$ is not a self-conjugate $t$-core, then by
 \cref{thm:oorthfact},
\[
\uoo_\lambda(1,\omega , \dots ,\omega^{t-1}) =  0.
\]
If it is, then again by the same result, we have  
\begin{equation*}
\uoo_\lambda(1,\omega , \dots ,\omega^{t-1}) =  (-1)^{\epsilon} 
\sgn (\sigma_{\lambda}) \, \prod_{i=0}^{\floor{{(t-2)}/{2}}} 
  \rs_{\lambda^{(i)},  \lambda^{(t-1-i)}}(1)  
  \times 
\begin{cases}
1 
& t \text{ even} ,\\
\uoo_{\lambda^{((t-1)/2)}}(1) & t \text{ odd},
\end{cases}  
\end{equation*}
where 
\[
\epsilon= \ds -\sum_{i=\floor{{(t+1)}/{2}}}^{t-1} \binom{n_{i}(\lambda)+1}2 
+ \begin{cases}
0 & t \text{ even},
\\r  & t \text{ odd},
\end{cases}
\]
$\ell(\lambda^{((t-1)/2)}) \leq 1$ and $\ell(\lambda^{(i)})+\ell(\lambda^{(t-1-i)}) \leq 2$ 
for 
$0 \leq i \leq \floor{{(t-2)}/{2}}$.
By definition,
\[
\uoo_{\lambda^{((t-1)/2)}}(1) = {h_{\lambda^{((t-1)/2)}}(1)+h_{\lambda^{((t-1)/2)}-1}(1)}=\begin{cases}
    1 & \ell(\lambda^{((t-1)/2)}) = 0, \\
    2 & \ell(\lambda^{((t-1)/2)}) > 0.
\end{cases}
\]
Therefore, using \eqref{rseval} completes the proof.
\end{proof}

Using \eqref{soneg}, we get the following result.

\begin{thm}
Suppose $\lambda$ is a partition in $\P_t$. 
Then $\uoo^-_\lambda(1,\omega, \dots ,\omega^{t-1})$ is nonzero if and only if
$\core \lambda t$ is a self-conjugate $t$-core and $\ell(\lambda^{(i)})+\ell(\lambda^{(t-1-i)}) \leq 1$ for all $0 \leq i \leq \floor{{(t-2)}/{2}}$.
In that case, if $\core \lambda t$ has rank $r$, 
\begin{equation*}
\uoo^-_\lambda(1,\omega, \dots ,\omega^{t-1}) =  (-1)^{|\lambda|+\epsilon} 
\sgn (\sigma_{\lambda})
 \times \begin{cases}
1 & t \text{ even},\\
1 & t \text{ odd and } \lambda^{((t-1)/2)}=\emptyset,
\\
2 & t \text{ odd and } \lambda^{((t-1)/2)} \neq \emptyset,
    \end{cases} 
\end{equation*}
where 
\[
\epsilon= \ds -\sum_{i=\floor{{(t+1)}/{2}}}^{t-1} \binom{n_{i}(\lambda)+1}2 
+ \begin{cases}
0 & t \text{ even},
\\r  & t \text{ odd}.
\end{cases}
\]
\end{thm}

\section{Classical characters at roots of unity}
\label{sec:class roots of unity}

In this section, we generalize \cref{thm:roots of unity} to all other classical characters. 
Unlike Schur polynomials or universal characters discussed in \cref{sec:univ roots of unity}, the 
values here are not bounded.

We will need the following elementary result.

\begin{lem}
\label{lem:eq}
Suppose $\lambda \in \P_{1}.$ We have 
\begin{enumerate}
    \item $\sp_{\lambda}(1)=s_{\lambda}(1,1)$,

    \item $\oo_{\lambda}(1)= s_{\lambda}(1,1)+s_{\lambda/(1)}(1,1) = s_{2\lambda}(1,1)$,

    \item $\oe_{\lambda}(1)= \begin{cases}
    2 & \lambda \neq \emptyset,\\
    1 & \lambda = \emptyset,
\end{cases}$

    \item $\oo_{\lambda}(-1) = (-1)^{|\lambda|}$.
\end{enumerate}

\end{lem}

\begin{proof}
    By \eqref{def:sp-jcob}, $\sp_{\lambda}(1) = h_{\lambda_1}(1,1) = s_{\lambda}(1,1)$,
 where the second equality follows from \eqref{def:SS-Jcob}. This proves \cref{lem:eq}(1). Now consider $\oo_{\lambda}(1)$. 
 By \eqref{def-univso}, we have 
 \begin{equation*}
          \oo_{\lambda}(1)
 = h_{\lambda_1}(1,1) + h_{\lambda_1-1}(1,1) = (\lambda_1+1)+(\lambda_1-1+1) = h_{2 \lambda}(1,1) = s_{2 \lambda}(1,1).
    \end{equation*}
 Similarly, evaluating $\oe_{\lambda}(1)$ and $\oo_{\lambda}(-1)$ using \eqref{def:even-Jcob} and \eqref{def-univso} prove \cref{lem:eq}(3) and \cref{lem:eq}(4) respectively. 
\end{proof} 

For the first result, we will appeal to \cref{thm:sympfact} at $\tX$, which also appeared as~\cite[Theorem 2.11]{ayyer-kumari-2022}.

\begin{thm}
Suppose $\lambda$ is a partition in $\P_t$. 
Then $\sp_{\lambda}(1,\omega, \dots ,\omega^{t-1})$ is nonzero if and only if
$\core \lambda t$ is a symplectic $t$-core.
In that case, if $\core \lambda t$ has rank $r$,
\begin{equation*}
    \begin{split}
\sp_{\lambda}(1,\omega, \dots ,\omega^{t-1}) =  (-1)^{\epsilon} 
\sgn (\sigma_{\lambda})\, \left(\lambda_1^{(t-1)}+1 \right) 
\prod_{i=0}^{(t-3)/2} & \left( \lambda_1^{(i)}+\lambda_1^{(t-2-i)}- \lambda_2^{(i)}-\lambda_2^{(t-2-i)}+1 \right) \\
& \times \begin{cases}
2\lambda_1^{(t/2-1)}+1 & t \text{ even},
\\1 & t \text{ odd},
    \end{cases} 
    \end{split}
\end{equation*}
where 
\[
\epsilon= \ds -\sum_{i=\floor{t/2}}^{t-2} \binom{n_{i}(\lambda)+1}2 + \begin{cases}
1+r & t \text{ even},
\\0  & t \text{ odd}.
\end{cases}
\]
\end{thm}

\begin{proof}
If $\core \lambda t$ is not a symplectic $t$-core, then by~\cref{thm:sympfact} at $\tX$, we get 
\begin{equation*}
    \sp_{\lambda}(1,\omega, \dots ,\omega^{t-1}) = 0.
\end{equation*}
If it is, then again by the same result, we have 
\begin{equation*}
\begin{split}
\sp_{\lambda}(1,\omega , \dots ,\omega^{t-1}) =  (-1)^{\epsilon} 
\sgn (\sigma_{\lambda}) \, 
\sp_{\lambda^{(t-1)}}(1) 
 \prod_{i=0}^{\floor{{(t-3)}/{2}}} 
 &s_{(\lambda^{(i)}, -\lambda^{(t-2-i)})_{2}}(1,1) \\ &\times 
\begin{cases}
\oo_{\lambda^{(t/2-1)}}(1) & t \text{ even,}\\
1 & t \text{ odd,}
\end{cases}  
\end{split}
\end{equation*}
where
\[
\epsilon= \ds -\sum_{i=\floor{t/2}}^{t-2} \binom{n_{i}(\lambda)+1}2 + \begin{cases}
1+r & t \text{ even},
\\0  & t \text{ odd}.
\end{cases}
\]
Now using the equalities given in \cref{lem:eq}, we get 
\begin{equation*}
\begin{split}
\sp_{\lambda}(1,\omega , \dots ,\omega^{t-1}) =  (-1)^{\epsilon} 
\sgn (\sigma_{\lambda}) \, s_{\lambda^{(t-1)}}(1,1) 
  \prod_{i=0}^{\floor{{(t-3)}/{2}}} & s_{{(\lambda^{(i)}, -\lambda^{(t-2-i)})_{2}}}(1,1)  \\
  &\times 
\begin{cases}
s_{2\lambda^{(t/2-1)}}(1,1) & t \text{ even},\\
1 & t \text{ odd}.
\end{cases}  
\end{split}
\end{equation*}
Since $s_{(\lambda_1,\lambda_2)}(1,1)=\lambda_1-\lambda_2+1$,
the result holds true.
\end{proof}

For the next result, we will appeal to \cref{thm:eorthfact} at $\tX$, which also appeared as~\cite[Theorem 2.15]{ayyer-kumari-2022}.

\begin{thm}
 Suppose $\lambda$ is a partition in $\P_t$. Then 
 $\oe_{\lambda}(1,\omega, \dots ,\omega^{t-1})$ is nonzero if and only if
 $\core \lambda t$ is an orthogonal $t$-core.
 In that case, if $\core \lambda t$ has rank $r$,
      \[
     \oe_{\lambda}(1,\omega, \dots ,\omega^{t-1})  = (-1)^{\epsilon} 
\sgn (\sigma_{\lambda}) \prod_{i=1}^{\floor{(t-1)/{2}}}
      \left( \lambda_1^{(i)}+\lambda_1^{(t-i)}- \lambda_2^{(i)}-\lambda_2^{(t-i)}+1 \right) 
\times \begin{cases}
    2 & \lambda^{(0)} \neq \emptyset,\\
    1 & \lambda^{(0)} = \emptyset,
\end{cases}  
      \]
   where 
\[
\epsilon= -\sum_{i=\floor{{(t+2)}/{2}}}^{t-1} 
\binom{n_{i}(\lambda)+1}{2} +r+ \begin{cases}
1+r & t \text{ even}, \\ 
0 & t \text{ odd}.
\end{cases}
\] 
\end{thm}

\begin{proof}
  If $\core \lambda t$ is not an orthogonal $t$-core, then by~\cref{thm:eorthfact} at $\tX$, we get 
  \[
  \oe_{\lambda}(1,\omega, \dots ,\omega^{t-1})  = 0.
  \]
If it is, then again by the same result, we have 
\begin{multline*}
\oe_{\lambda}(1,\omega, \dots ,\omega^{t-1})  =  (-1)^{\epsilon} 
\sgn (\sigma_{\lambda}) \,
  \oe_{\lambda^{(0)}}(1)   \prod_{i=1}^{\floor{{(t-1)}/{2}}} s_{(\lambda^{(i)}, -\lambda^{(t-i)})_2}(1,1) \\
  \times \begin{cases}
(-1)^{ \lambda_1^{(t/2)}}\displaystyle {\oo_{\lambda^{(t/2)}}(-1)} & t \text{ even},\\
1 & t \text{ odd},
\end{cases}   
\end{multline*}
where 
\[
\epsilon= -\sum_{i=\floor{{(t+2)}/{2}}}^{t-1} 
\binom{n_{i}(\lambda)+1}{2} +r+ \begin{cases}
1+r & t \text{ even}, \\ 
0 & t \text{ odd}.
\end{cases}
\] 
Using the equalities given in \cref{lem:eq}, we have 
\begin{equation*}
    \oe_{\lambda}(1,\omega, \dots ,\omega^{t-1})  =
    (-1)^{\epsilon} 
\sgn (\sigma_{\lambda}) \, \prod_{i=1}^{\floor{{(t-1)}/{2}}} s_{{(\lambda^{(i)}, -\lambda^{(t-i)})_{2}}}(1,1)  
\times \begin{cases}
    2 & \lambda^{(0)} \neq \emptyset,\\
    1 & \lambda^{(0)} = \emptyset.
\end{cases}  
\end{equation*}
Since $s_{(\lambda_1,\lambda_2)}(1,1)=\lambda_1-\lambda_2+1$, this completes the proof.
\end{proof}

For the final result, we will appeal to \cref{thm:oorthfact} at $(\tX, 1)$, which also appeared as~\cite[Theorem 2.17]{ayyer-kumari-2022}.

\begin{thm}
   Suppose $\lambda$ is a partition in $\P_t$. Then
   $\oo_{\lambda}(1,\omega , \dots ,\omega^{t-1})$ is nonzero if and only if
   $\core \lambda t$ is a self-conjugate $t$-core.
   In that case, if $\core \lambda t$ has rank $r$,
        \begin{equation*}
            \begin{split}
             \oo_{\lambda}(1,\omega , \dots ,\omega^{t-1}) = (-1)^{\epsilon} \sgn(\sigma_{\lambda}) \prod_{i=0}^{\floor{{(t-2)}/{2}}} & 
             \left( \lambda_1^{(i)}+\lambda_1^{(t-1-i)}- \lambda_2^{(i)}-\lambda_2^{(t-1-i)}+1 \right) \\
       & \times \begin{cases}
2\lambda_1^{({(t-1)}/{2})}+1 & t \text{ even},
\\1 & t \text{ odd},
    \end{cases}    
            \end{split}
        \end{equation*}
  where \[
\epsilon= \ds -\sum_{i=\floor{{(t+1)}/{2}}}^{t-1} 
\binom{n_{i}(\lambda)+1}2 
+ \begin{cases}
0 & t \text{ even},
\\r  & t \text{ odd}.
\end{cases}
\]      
\end{thm}

\begin{proof}
If $\core \lambda t$ is not a self-conjugate $t$-core, then by~\cref{thm:oorthfact} at $(\tX, 1)$, we have 
        \[
        \oo_{\lambda}(1,\omega , \dots ,\omega^{t-1})  =0.
        \]
If it is, then again by the same result, we have 
\[
\oo_{\lambda}(1,\omega , \dots ,\omega^{t-1})  =  (-1)^{\epsilon} 
\sgn (\sigma_{\lambda})  \, 
\prod_{i=0}^{\floor{{(t-2)}/{2}}}   s_{(\lambda^{(i)}, -\lambda^{(t-1-i)})_{2}}(1,1) 
\times 
\begin{cases}
\oo_{\lambda^{(t-1)/2}}(1) & t \text{ odd},\\
1 & t \text{ even},
\end{cases}   
\]
 where \[
\epsilon= 
\ds -\sum_{i=\floor{{(t+1)}/{2}}}^{t-1} 
\binom{n_{i}(\lambda)+1}2 
+ \begin{cases}
0 & t \text{ even},
\\r  & t \text{ odd}.
\end{cases}
\]
Using the equalities given in \cref{lem:eq}, we have 
\[
\oo_{\lambda}(1,\omega , \dots ,\omega^{t-1}) =  (-1)^{\epsilon} 
\sgn (\sigma_{\lambda})  \, 
\prod_{i=0}^{\floor{{(t-2)}/{2}}}   s_{{(\lambda^{(i)}, -\lambda^{(t-1-i)})_{2}}}(1,1) 
\times 
\begin{cases}
s_{2\lambda^{((t-1)/2)}}(1,1) & t \text{ odd},\\
1 & t \text{ even}.
\end{cases}
\]
Since $s_{(\lambda_1,\lambda_2)}(1,1)=\lambda_1-\lambda_2+1$, this completes the proof.
\end{proof}

\section*{Acknowledgements}
We thank Dipendra Prasad for encouragement as well as many useful conversations, and Seamus Albion for discussions.
{We also thank the referees for their many useful comments.}
The authors were partially supported by SERB Core grant CRG/2021/001592 
as well as the DST FIST program - 2021 [TPN - 700661].

\bibliographystyle{alpha}
\bibliography{Bibliography}

\end{document}